\documentclass[10pt]{article}
\usepackage[top=1in, bottom=1in, left=1.25in, right=1.25in]{geometry}

\usepackage{graphicx}
\usepackage{epstopdf}
\usepackage{epsfig}
\usepackage{url}
\usepackage{geometry}
\usepackage{color}
\usepackage{amsmath}
\usepackage{amssymb,amsthm}
\usepackage[ruled,vlined,linesnumbered]{algorithm2e}
\usepackage{chngcntr}

\newtheorem{theorem}{Theorem}
\newtheorem{lemma}{Lemma}
\newtheorem{definition}{Definition}

\newcommand{\range}{\text{Range}}
\newcommand{\mat}[1]{\boldsymbol{#1}}
\newcommand{\proj}[1]{\boldsymbol{\mathcal{P}}_{\boldsymbol{#1}}}
\newcommand{\mb}[1]{\mathbb{#1}}
\newcommand{\mc}[1]{\mathcal{#1}}
\newcommand{\normf}[1]{\left\|#1 \right\|_\text{F}}
\newtheorem{remark}{Remark}
\counterwithin{figure}{subsection}

\definecolor{green}{rgb}{0,.5,.5}

\title{A Randomized Tensor Singular Value Decomposition based on the t-product\thanks{This research is based upon work partially supported by the Office of the Director of National Intelligence (ODNI), Intelligence Advanced Research Projects Activity (IARPA), via IARPA’s 2014-14071600011 and by the National Science Foundation under NSF 1319653. The views and conclusions contained herein are those of the authors and should not be interpreted as necessarily representing the official policies or endorsements, either expressed or implied, of ODNI, IARPA, or the U.S. Government. The U.S. Government is authorized to reproduce and distribute reprints for Governmental purpose notwithstanding any copyright annotation thereon.}} 

\author{Jiani Zhang\footnotemark[2]\ 
\and Arvind K. Saibaba\footnotemark[3]
\and Misha E. Kilmer\footnotemark[2]\ 
\and Shuchin Aeron\footnotemark[4]}

\begin{document}
\maketitle

\renewcommand{\thefootnote}{\fnsymbol{footnote}}

\footnotetext[2]{Department of Mathematics, Tufts University, Medford, MA 02155 ({jiani.zhang@tufts.edu, misha.kilmer@tufts.edu}).}
\footnotetext[3]{Department of Mathematics, North Carolina State University, Raleigh, NC 27695 ({asaibab@ncsu.edu}).}
\footnotetext[4]{Department of Electrical and Computer Engineering, Tufts University, Medford, MA 02155 ({shuchin@ece.tufts.edu}).}

\renewcommand{\thefootnote}{\arabic{footnote}}
\newcommand{\arvind}[1]{{\color{red} #1}}
\newcommand{\JZ}[1]{\textcolor{green}{{#1}}}
\newcommand\MEK[1]{\textcolor{blue}{{#1}}}


\begin{abstract}
The tensor Singular Value Decomposition (t-SVD) for third order tensors that was proposed by Kilmer and Martin~\cite{2011kilmer} has been applied successfully in many fields, such as computed tomography, facial recognition, and video completion. In this paper, we propose a method that extends a well-known randomized matrix method to the t-SVD.  This method can produce a factorization with similar properties to the t-SVD, but is more computationally efficient on very large datasets. We present details of the algorithm, theoretical results, and provide numerical results that show the promise of our approach for compressing and analyzing datasets. We also present an improved analysis of the randomized subspace iteration for matrices, which may be of independent interest to the scientific community.  

\end{abstract}



\pagestyle{myheadings}
\thispagestyle{plain}
\markboth{J. Zhang, A. K. Saibaba, M. E. Kilmer, S. Aeron}{A Randomized Tensor Singular Value Decomposition}

\section{Introduction} 
In this era of ``big data,'' it is not uncommon for the size of a matrix operator, or a dataset, to reach the scale of petabytes or even exabytes.  By 2013, for example, Facebook was already claiming to use $1.5$ petabytes to store about $10$ billion photos, and Netflix claimed to use $3.14$ petabytes to store available shows and movies~\cite{vance2013netflix}. 
As another example, the size of the matrix operator in quantum chromodynamics is on the order of several millions, or even billions~\cite{frommer2012numerical}.   On the one hand, there still seems to be a push to obtain ever more information by collecting more data since the storage capability exists, and for modeling very fine scale phenomena.  
On the other hand, current data analysis and scientific computing methods are continually challenged by the expanding sizes of the models and datasets.

Almost all of the methods in data analysis and scientific computing rely on matrix algorithms \cite{witten2013randomized}. In particular, the low-rank matrix approximation,
\begin{equation}
	\mathbf{A}_{m \times n} \approx \mathbf{B}_{m \times k} \mathbf{C}_{k \times n},
\end{equation}
where $k < \text{min} \lbrace m,n \rbrace$, is used often, because it allows us to store or analyze the matrix $\mathbf{A}$ by the factor matrices $\mathbf{B}$ and $\mathbf{C}$ instead of the full matrix, which is more efficient when $k \ll \text{min} \lbrace m,n \rbrace$. Moreover, these smaller matrices have often been shown to provide specific structure that help analyze the data with better results, see \cite{berry2007algorithms,chan1994low,van1996schur}.  

It is well known that truncating the matrix singular value decomposition to $k$ terms provides the optimal rank-k approximation to a matrix in both the matrix 2-norm and Frobenius norm, and the algorithms for computing the approximation are numerically robust.   Therefore, it is not surprising that the truncated SVD has been proposed for use in many applications, included but not limited to image processing \cite{sadek2012svd,hansen2006deblurring}, statistics \cite{hastie2009elements,tenorio2001statistical} and Partial Differential Equations (PDEs) \cite{dorobantu1998wavelet}. 
However, the cost of accurately computing the truncated matrix SVD can be prohibitively expensive, making it unsuitable for very large scale applications~\cite{simon2000low}.
 
 Therefore, much work has been devoted to generating low-rank approximations which have similar rank-revealing properties to the SVD but which are cheaper to compute.   As a trade-off, one gives up the optimality property that is 
 the signature feature of the SVD.   
In recent years, much work has been devoted to the development of randomized algorithms for computing low-rank matrix approximations.  They are particularly appealing because although the cannot give the optimal low-rank approximation, they can be shown to give nearly optimal results.   

Randomized algorithms have been recently developed for accurate low-rank representations, see~\cite{frankl1988johnson,ailon2006approximate,dasgupta2003elementary,achlioptas2003database,ailon2009fast,2011halko,halko2011algorithm,clarkson2009} and several others. Randomized matrix methods are powerful because they are numerically robust, computationally efficient, and suitable for implementation on a variety of architectures, including high-performance computing. They usually come in two different flavors -- based on random sampling of columns and rows of the matrix, or by random projection onto a lower dimensional subspace.  More details of these two different approaches can be found in the review paper by~\cite{mahoney2011randomized}.

All the previous work referenced above involves generating near optimal low-rank {\bf matrix} approximations from 
randomized techniques. However, many data sets and matrix operators are inherently multidimensional in nature.   Consider, for example, a collection of hyperspectral images.   One can choose to scan the image at each wavelength as a vector, resulting in a matrix representation of all the hyperspectral data.   But one might also store each 2D image as a slice of a 3-way array, resulting in a third order tensor representation of the data.  As another example, each frame of a color image is technically a 3D array, and so the time sequence of a video can be stored as a 4D array with time as the last index -- the resulting data structure is called a fourth order tensor.    The question of interest to us in this paper is, if we choose to keep the multi-way structure inherent in the data, can we then generalize the concept of a best (in some norm) ``low-rank'' approximation to tensor data, and if so, how do we move from randomized low-rank matrix techniques to randomized tensor factorizations.    

First, one must decide on method of tensor decomposition, and with it, the notion of best ``low-rank''.  The well-known CP decomposition, originally proposed by Hitchcock in 1927 \cite{hitchcock1927tensor}, is a decomposition as a sum of multiway outer products of vectors.  Tensor rank is then defined in terms of the minimal sum of these outer products necessary to construct the tensor \cite{kolda2009tensor}.  The difficulty is a best rank-k approximation need not exist without extra assumptions, and computing the rank-k approximation is also highly non-trivial even when it does exist. The Tucker decomposition \cite{tucker1963implications}, developed by Tucker in 1963, is an alternative decomposition.  For a third-order tensor, the Tucker3 decomposition requires 3 factor matrices and a core tensor.  A Tucker3 factorization always exists, and can be generated such that factor matrices can have orthonormal columns -- the HOSVD \cite{de2000multilinear} is such a decomposition.  One can specify the rank of the factor matrices, and thus obtain a so-called best rank-$(R_1,R_2,R_3)$ Tucker3 approximation.   However, unlike the matrix case, this best approximation cannot be obtained by truncating the full HOSVD.     

A more recent alternative approach to factoring tensors was introduced by Kilmer and Martin in 2011 \cite{2011kilmer}.  In their work, the authors present the concept of a tensor-tensor product with suitable algebraic structure such that classical matrix-like factorizations are possible.   In particular, they give the definition of the tensor SVD (t-SVD) over this new product, and show that truncating that expansion does give a compressed result that is optimal in the Frobenius norm. 

Applications of all three types of these tensor-based decompositions can be found in the literature:  see for example \cite{vasilescu2002multilinear, smilde2005multi,dunlavy2011multilinear,zhang2014novel,hao2014nonnegative,ely20135d,semerci2013tensor}.  Though the decompositions are different, the common theme among the results is evidence that tensor-based decompositions of the data/operators provide considerable improvement over the matrix-based counterparts.  However, all these tensor-based decompositions are deterministic, so it is natural to explore randomization of these tensor decompositions as well.   
 
To the best of our knowledge, the authors in \cite{drineas2007randomized} appear to have pioneered the generalization of random sampling methods to tensors. Specifically, they extended random sampling methods to the Tucker decomposition, and provided a guide to the theoretical analysis for the tensor-based decomposition via random sampling methods. In \cite{tsourakakis2010mach}, the authors provide numerical examples of Tucker decomposition with the random sampling method.   A literature search also reveals attempts to extend the random sampling approach to tensors-based on CP decomposition and Tucker decomposition, \cite{biagioni2015randomized,sigurdson2012randomized}.

In this paper, we extend a well-known matrix-based random projection method, the randomized SVD (r-SVD)~\cite{2011halko}, to third-order tensors through use of the algebra induced by the t-product and the t-SVD~\cite{2011kilmer}. The motivation for focusing our efforts on the randomization of the t-SVD is the theoretical and computational advantages provided by the t-product, as well as the use of the t-SVD in applications.
For example, as mentioned above, a best $k$-term expansion can be obtained from truncation of the t-SVD.  Further, the t-SVD computations are readily parallelizable.  Under the t-product, there are well defined concepts of orthogonality, identity and orthogonal projections, QR factorizations, and the like~\cite{kilmer2013third,gleich2013power,martin2013order}.
Moreover, in some applications, such as compression and facial recognition, the t-SVD has been shown to have superior compression characteristics \cite{hao2013facial} relative to the Tucker decomposition. When they use some storage, the t-SVD has better performance in term of recognition rate.
\paragraph{Contributions} We develop randomized algorithms for low-rank decompositions of tensors, based on the t-product. The first algorithm applies the randomized SVD to the frontal slices of the tensor (in the Fourier domain), where as the second algorithm applies the randomized power method to the same slices. Efficient implementations of the above algorithms are also provided in the spatial, as well as Fourier domains. We develop a framework for error analysis and derive expressions for expected behavior of the error in the low-rank representation, as well as probabilistic bounds for deviation from expectation. Our analysis for randomized power method is novel even for the matrix case. Application to facial recognition, including a parallel implementation, underscores the benefits of the proposed methods.

This paper is organized as follows. In section \ref{preliminaries}, we review some relevant mathematical concepts including the matrix r-SVD, basic definitions and theorems of tensors, and the t-SVD based on the t-product. In section \ref{rt-svd}, we give the basic randomized t-SVD (rt-SVD) method and extend this to the rt-SVD with subspace iteration.  There, we also provide the analysis of error expectations. In section \ref{Numerical}, we compare the errors of our algorithms with theoretical minimal errors on a real dataset, apply the algorithms in the application of facial recognition, and compute them in parallel on a cluster to show the improvement of computation efficiency.   Conclusions and future work are provided in section \ref{conclusion}.   Where noted in the body of the paper, some of the proof details are provided in the Appendix \ref{s_proof}.

\section{Preliminaries}\label{preliminaries}
In this section, we introduce some definitions and algorithms which are used throughout the paper.  We begin by noting that boldface lowercase letters indicate vectors, e.g. $\mathbf{a}$. Boldface uppercase letters indicate matrices, e.g. $\mathbf{A}$.  Boldface Euler script letters indicate tensors, e.g. $\mathcal{A}$, and unless otherwise specified, the tensors are third order.   

We will assume that $\mathbf{A}$ is of rank $r$, and $\mathbf{A} = \sum_{i=1}^{r} \sigma_i \mathbf{u}_i \mathbf{v}_i^{\rm H}$ is the singular value decomposition of $\mathbf{A}$, with $\sigma_1 \ge \sigma_2 \ge \cdots \sigma_r > 0$.

\subsection{The r-SVD}

The randomized Singular Value Decomposition (referred to as r-SVD), was proposed in a series of papers published over the last decade (see e.g.,~\cite{liberty2007randomized,woolfe2008fast}) and popularized by the review paper~\cite{2011halko}.   
The first step in computing the r-SVD is generating several Gaussian random vectors $\mat{W} \in \mb{R}^{n\times(k+p)}$ that are, with high probability, linearly independent.   Here $k$ is the desired target truncation term of the approximation, and $p$ is a non-negative integer oversampling parameter.   
The matrix  $\mat{Y} :=  \mat{AW} \in \mathbb{C}^{m \times (k+p)}$ thus contains random linear combinations of the columns of  $\mat{A}$. 
 
A thin QR of $\mat{Y}$ is computed, so that $\mbox{range}(\mat{Y}) = \mbox{range}(\mathbf{Q})$.   

The idea is if $\mathbf{A}$ has rapidly decaying singular values, so that the dominant part of the range of $\mathbf{A}$ is marked by the first $k$ or so left singular vectors, i.e., $\mathbf{A} \approx \mathbf{Q}\mathbf{Q}^{\rm H}\mathbf{A}$.  

Thus, one computes $\mathbf{B} := \mathbf{Q}^{\rm H}\mathbf{A}$ followed by 
the compact SVD of $\mathbf{B}$, $\mathbf{B} = \tilde{\mathbf{U}} \tilde{\mathbf{S}} \tilde{\mathbf{V}}^{\rm H}$.   The estimated desired singular values of $\mathbf{A}$ are the diagonals of $\tilde{\mathbf{S}}$, while $\mathbf{Q} {\tilde{\mathbf{U}}}$ gives the estimated right singular vectors of $\mathbf{A}$.  
Algorithm $\ref{alg_rSVD}$ summarizes the procedure described above. 

\begin{algorithm}[H]\label{alg_rSVD}
\SetAlgoLined
\SetKwInOut{Input}{Input}\SetKwInOut{Output}{Output}
\Input{$\mathbf{A} \in \mathbb{C}^{m \times n}$, target truncation term $k$, and oversampling parameter $p$}
\Output{$\mathbf{U}_{k} \in \mathbb{C}^{n \times k}$, $\mathbf{S}_{k} \in \mathbb{C}^{k \times k}$, and $\mathbf{V}_{k} \in \mathbb{C}^{n \times k}$}
\BlankLine
  Generate a Gaussian random matrix $\mathbf{W} \in \mathbb{R}^{n \times (k+p)}$\;
  Form a matrix $\mathbf{Y} = \mathbf{A} \mathbf{W}$\;
  Construct matrix $\mathbf{Q} \in \mathbb{C}^{n \times (k+p)}$ which is orthogonal column basis for $\mathbf{Y}$\;
  Form $\mathbf{B} \in \mathbb{C}^{(k+p) \times n}$, $\mathbf{B} = \mathbf{Q}^{*} \mathbf{A}$\;
  Compute $\mathbf{B} = \tilde{\mathbf{U}} \tilde{\mathbf{S}} \tilde{\mathbf{V}}^{\rm H}$,\;
  Set $\mathbf{U} = \tilde{\mathbf{U}}( :,1\!:\!k)$, $\mathbf{S}_{k} = \tilde{\mathbf{S}}(1\!:\!k,1\!:\!k)$, $\mathbf{V}_{k} = \tilde{\mathbf{V}}(:,1\!:\!k)$\;
  Form $\mathbf{U}_{k} = \mathbf{Q}_{k} \mathbf{U}$.
\caption{r-SVD method~\cite{2011halko}}
\end{algorithm}

When $\mathbf{A}$ is dense and of size $n \times n$, this algorithm can take $\mc{O}(kn^{2})$ flops. For more details, see~\cite{woolfe2008fast}. The expected error in the low-rank approximation measured using the Frobenius norm can be bounded, as the result below shows. This result was first proved in~\cite[Theorem 10.5]{2011halko}, but is stated here in a slightly different form. The original result derived the error bound for $\mathbb{E} \normf{\mathbf{A}-\mathbf{Q}\mathbf{Q}^{\rm H}\mathbf{A}}$, whereas, in the next section we require the error bound for $\mathbb{E} \normf{\mathbf{A}-\mathbf{Q}\mathbf{Q}^{\rm H}\mathbf{A}}^2$.  

\begin{theorem} Given a matrix $\mathbf{A}\in \mathbb{C}^{m \times n}$ and a Gaussian random matrix $\mathbf{W}\in \mathbb{R}^{n \times (k+p)}$, let $p \geq 2$ be a pre-specified integer. Suppose that $\mat{Q}$ is computed as in Algorithm~\ref{alg_rSVD}, then the expected approximation error is 
\begin{equation}\label{expected_error}
\mathbb{E} \normf{\mathbf{A}-\mathbf{Q}\mathbf{Q}^{\rm H}\mathbf{A}}^2 \> \leq \>  \left(1+\dfrac{k}{p-1}\right)\left(\sum_{j=k+1}^{\min\{m,n\}} \sigma^{2}_{j}\right)
\end{equation}
where $\sigma_{j}$ is the $j^{th}$ singular value of $\mathbf{A}$.
\end{theorem}
\begin{proof}
The proof follows readily from~\cite[Theorem 10.5]{2011halko}.
\end{proof}

From the inequality \eqref{expected_error}, the value of $\mathbb{E} \| \mathbf{A}-\mathbf{Q}\mathbf{Q}^{\rm H}\mathbf{A} \|_{\rm F}^2$ depends on $\sum_{j>k} \sigma^{2}_{j}$. When the singular values of $\mathbf{A}$ decay gradually, $\sum_{j>k} \sigma^{2}_{j}$ can be large, and therefore the low-rank approximation as computed above may not be sufficiently accurate. In this situation, Algorithm~\ref{alg_rSVD_power_subspace}, which is based on subspace iteration, may be preferred.

\begin{algorithm}[H]\label{alg_rSVD_power_subspace}
\SetAlgoLined
\SetKwInOut{Input}{Input}\SetKwInOut{Output}{Output}
\Input{$\mathbf{A} \in \mathbb{C}^{m \times n}$, target truncation term $k$, a parameter $q$, and an oversampling parameter $p$}
\Output{An orthogonal column basis $\mathbf{Q}$ of $\mathbf{Y}$}
\BlankLine
 Generate a Gaussian random matrix $\mathbf{W} \in \mathbb{R}^{n \times (k+p)}$\;
 Form a matrix $\mathbf{Y_{0}} = \mathbf{A} \mathbf{W}$ and compute the QR factorization of $\mathbf{Y}_{0}=\mathbf{Q}_{0}\mathbf{R}_{0}$\;

 \For{$i\leftarrow 1$ \KwTo $q$}{
      Form $\mathbf{\tilde{Y}}_{i} = \mathbf{A}^{\rm H} \mathbf{Q}_{i-1}$ and compute the QR factorization of $\mathbf{\tilde{Y}}_{i}=\mathbf{\tilde{Q}}_{i}\mathbf{\tilde{R}}_{i}$\;
      Form $\mathbf{Y}_{i} = \mathbf{A} \mathbf{\tilde{Q}}_{i}$ and compute the QR factorization of $\mathbf{Y}_{i}=\mathbf{Q}_{i}\mathbf{R}_{i}$\; 
   }
Form a matrix $\mathbf{Q}=\mathbf{Q}_{q}$\;
\caption{r-SVD method with subspace iteration~\cite{2011halko}}
\end{algorithm}

The error analysis for Algorithm~\ref{alg_rSVD_power_subspace} is developed in~\cite{2011halko} for the spectral norm, but no analysis was presented in the Frobenius norm. We present the following result that characterizes the error due to Algorithm~\ref{alg_rSVD_power_subspace} in the Frobenius norm. We assume that $k$ is the target truncation term, and define the singular value gap $\tau_k  \equiv \frac{\sigma_{k+1}}{\sigma_{k}}$. We further assume that $\tau_k \ll 1$, i.e., there is a gap between singular values $k$ and $k+1$.

\begin{theorem}[Average Frobenius error for Algorithm~\ref{alg_rSVD_power_subspace}]\label{t_gu_subspace}
Let $\mathbf{A}\in\mb{C}^{m\times n}$ and let $\mathbf{W} \in \mb{R}^{n\times (k+p)}$ be a Gaussian random matrix with $p \geq 2$ being the oversampling parameter. Suppose $\mathbf{Q}$ is obtained from Algorithm~\ref{alg_rSVD_power_subspace}, then
\begin{equation*}
\mathbb{E}\normf{ \mathbf{A}-\mathbf{Q}\mathbf{Q}^{\rm H}  \mathbf{A}}^2 \> \leq
\> \left(1 + \frac{k}{p-1}\tau_k^{4q} \right) \left(\sum_{j = k+1}^{\min\{m,n\}} \sigma_j^2\right) ,
\end{equation*}
where $k$ is a target truncation term, $q$ is the number of iterations, $\sigma_{j}$ is the $j^{th}$ singular value of $\mathbf{A}$, and $\tau_{k}= \sigma_{k+1}/\sigma_k \ll 1$ is the singular value gap. 
\end{theorem}
\begin{proof}
See Appendix~\ref{s_proof}. 
\end{proof}

The error due to the randomized subspace iteration is similar to Theorem~\ref{expected_error}, except for the term $\tau_k^{4q}$. As the number of subspace iterations $q$ increases, the effect of the residual term $k\tau_k^{4q}/(p-1)$ decreases, and the subspace iteration achieves the optimal error of the SVD. Also note that for $q = 0$, we exactly obtain the result in Theorem~\ref{expected_error}. A similar result was presented in~\cite[Theorem 5.7]{gu2015subspace}, but our analysis is sharper, see discussion in Remark~\ref{r_gucomp}.

\subsection{Tensors}\label{definitions}
A tensor is a multi-dimensional array, and the order of the tensor is the number of dimensions of this array. In this paper, we focus on the third order tensor $\mathcal{A} \in \mathbb{R}^{n_{1} \times n_{2} \times n_{3}}$. 
Each entry of the tensor $\mathcal{A}$ is denoted by Matlab indexing notation, i.e., $\mathcal{A}(i,j,k)$. 

A fiber of tensor $\mathcal{A}$ is a one-dimensional array defined by fixing two indices. $\mathcal{A}(:,j,k)$ is the $(j,k)^{th}$ column fiber, $\mathcal{A}(i,:,k)$ is the $(i,k)^{th}$ row fiber, and $\mathcal{A}(i,j,:)$ is the $(j,k)^{th}$ tube fiber. A slice of tensor $\mathcal{A}$ is a two-dimensional array defined by fixing one index. $\mathcal{A}(i,:,:)$ is the $i^{th}$ horizontal slice, $\mathcal{A}(:,j,:)$ is the $j^{th}$ lateral slice, and $\mathcal{A}(:,:,k)$ is the $k^{th}$ frontal slice. For convenience, $\mathcal{A}(:,:,k)$ is written as $\mathcal{A}^{(k)}$.  A third order tensor $\mathcal{A}$ can be seen as an $n_{1} \times n_{2}$ array of tube fibers, each of size $1 \times 1 \times n_{3}$. The t-product of two tube fibers is defined as their the circular convolution, so the t-product between two tensors can be defined as in Definition~\ref{def_t-product}. 
Next, we review several definitions from~\cite{2011kilmer} that will be necessary for the rest of this paper.

\begin{figure*}[h]\label{tensors}
\begin{center}
\includegraphics[scale=.35]{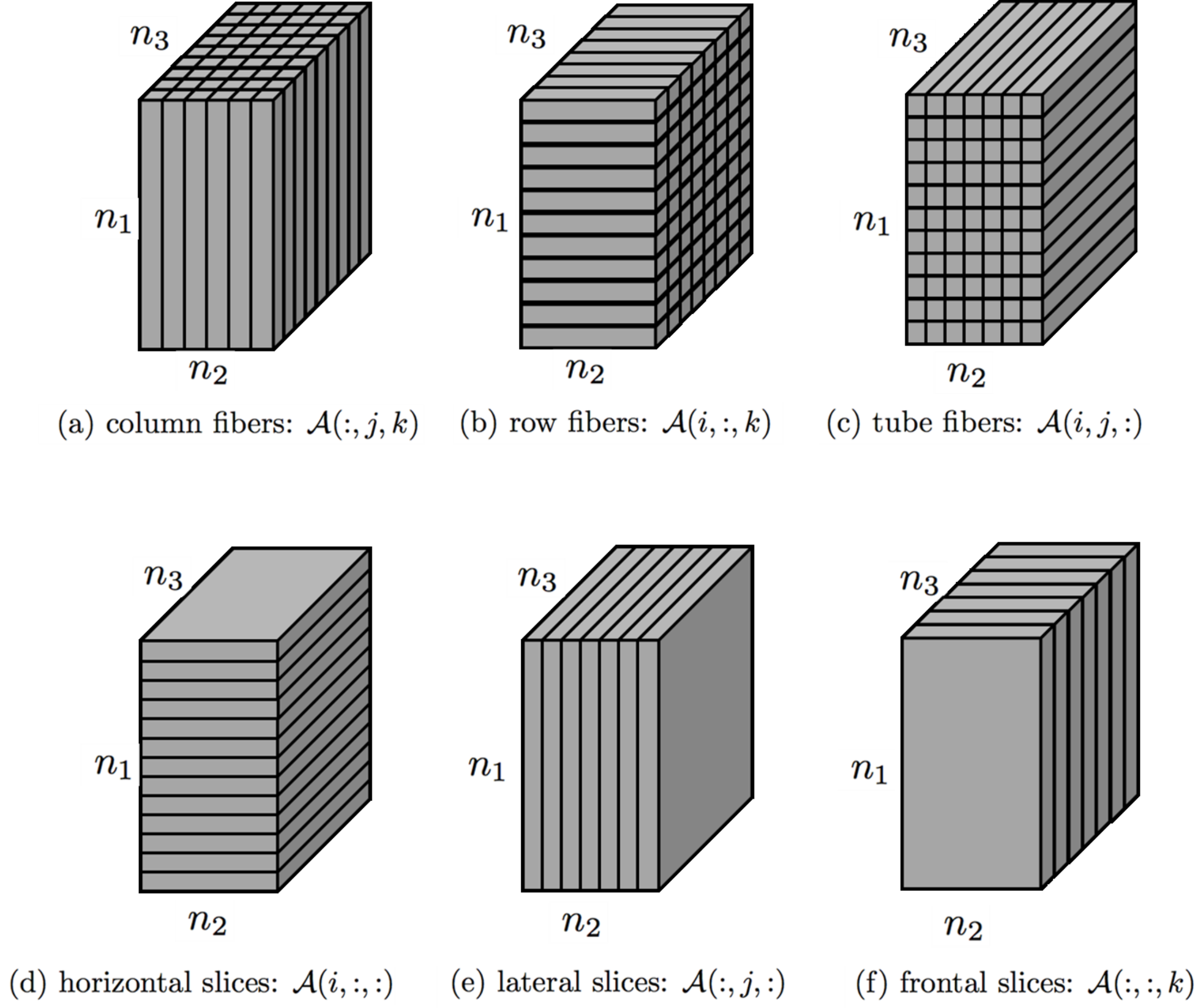}
\end{center}
\caption{Fibers and slices of an $n_{1} \times n_{2} \times n_{3}$ tensor $\mathcal{A}$}.
\end{figure*}

\begin{definition}[t-product]\label{def_t-product}
Let $\mathcal{A}$ be an $n_{1} \times n_{2} \times n_{3}$ tensor and $\mathcal{B}$ be an $n_{2} \times n_{4} \times n_{3}$ tensor. The t-product of $\mathcal{A}$ and $\mathcal{B}$, $\mathcal{C} = \mathcal{A} \ast \mathcal{B}$, is an $n_{1} \times n_{4} \times n_{3}$ tensor
\begin{equation*}
	\mathcal{C}(i,j,:)  = \sum_{k =1}^{n_{2}}\mathcal{A}(i,k,:) \ast \mathcal{B}(k,j,:) = \sum_{k =1}^{n_{2}}\mathcal{A}(i,k,:) \circ \mathcal{B}(k,j,:)       
	\end{equation*}                                                  where the notation $\circ$ denotes the circular convolution.
\end{definition}

Because the circular convolution of two tube fibers can be computed by discrete Fourier transform, the t-product can be alternatively computed in the Fourier domain\footnotemark, as shown in Algorithm \ref{alg_t-product}. 
\footnotetext{All quantities in the spatial domain are real and all quantities in the Fourier domain may be complex.}
\begin{algorithm}[H]\label{alg_t-product}
\SetAlgoLined
\SetKwInOut{Input}{Input}\SetKwInOut{Output}{Output}
\Input{$\mathcal{A} \in \mathbb{R}^{n_{1} \times n_{2} \times n_{3}}$ and $\mathcal{B} \in \mathbb{R}^{n_{2} \times n_{4} \times n_{3}}$ }
\Output{an $n_{1} \times n_{4} \times n_{3}$ tensor $\mathcal{C}$, $\mathcal{C} = \mathcal{A} \ast \mathcal{B}$}
\BlankLine
 $\hat{\mathcal{A}} \leftarrow \tt{fft}(\mathcal{A},[\,],3)$\;
 $\hat{\mathcal{B}} \leftarrow \tt{fft}(\mathcal{B},[\,],3)$\;
\For{$i\leftarrow 1$ \KwTo $n_{3}$}{
  $\hat{\mathcal{C}}^{(i)} = \hat{\mathcal{A}}^{(i)}\hat{\mathcal{B}}^{(i)}$ \;
   }
 $\mathcal{C} \leftarrow \tt{ifft}(\hat{\mathcal{C}},[\,],3)$  
\caption{t-product computation in Fourier domain~\cite{2011kilmer}}
\end{algorithm}

\begin{definition}[Identity tensor]
	The $n_{1} \times n_{2} \times n_{3}$ identity tensor $\mathcal{I}$ is the tensor whose first frontal slice is the $n_{1} \times n_{2}$ identity matrix, and whose other frontal slices are all zeros.  
\end{definition}

\begin{definition}[Transpose]
	If $\mathcal{A}$ is an $n_{1} \times n_{2} \times n_{3}$ tensor, then $\mathcal{A}^{\rm T}$ is an $n_{2} \times n_{1} \times n_{3}$ tensor obtained by transposing each of the frontal slices and then reversing the order of transposed frontal slices $2$ through $n_{3}$, see Fig. \ref{transpose}.   
\end{definition}

\begin{figure*}[h]
\begin{center}
\includegraphics[scale=.35]{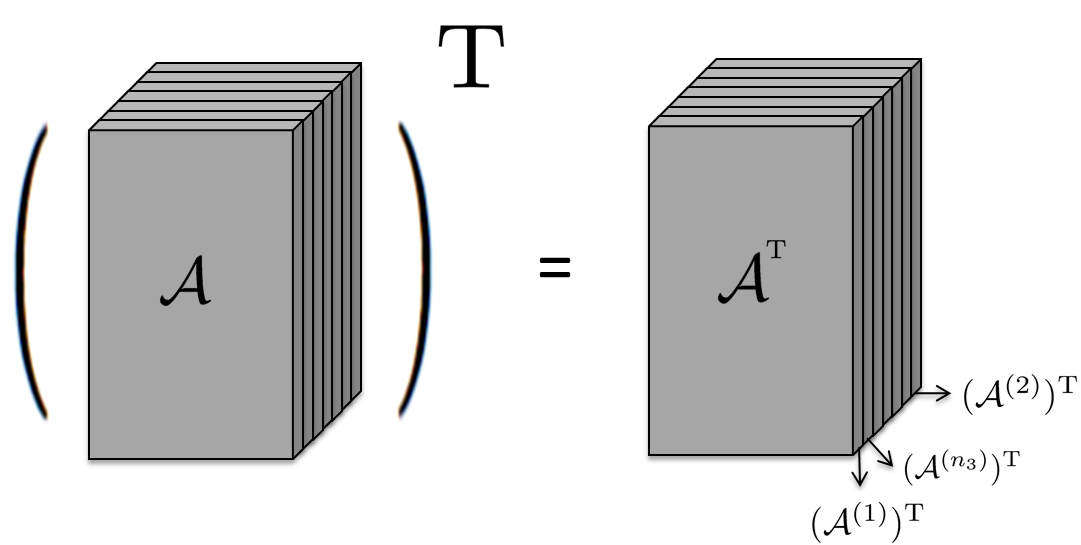}
\end{center}
\caption{Transpose of an $n_{1} \times n_{2} \times n_{3}$ tensor $\mathcal{A}$}.\label{transpose}
\end{figure*}

\begin{definition}[Orthogonality]
	An $n_{1} \times n_{2} \times n_{3}$ tensor $\mathcal{A}$ is called orthogonal if $n_1 = n_2$, if the t-product of $\mathcal{A}^{\rm T}$ and $\mathcal{A}$ is equal to the identity tensor, i.e.,
	\begin{equation*}
		\mathcal{A}^{\rm T} \ast \mathcal{A} = \mathcal{I}_{n_2n_2n_3}.
	\end{equation*}
If the above equation holds but $n_1 > n_2$, then the tensor is said to be partially orthogonal.
\end{definition}

\begin{definition}[f-Diagonal]
	An $n_{1} \times n_{2} \times n_{3}$ tensor $\mathcal{A}$ is called f-diagonal, if each frontal face of $\mathcal{A}$ is diagonal.
\end{definition}

\begin{definition}[t-QR factorization]
	Given an  $n_{1} \times n_{2} \times n_{3}$ tensor $\mathbf{A}$, the t-QR factorization of $\mathcal{A}$ is
	\begin{equation*}
		\mathcal{A} =\mathcal{Q} \ast \mathcal{R}
	\end{equation*}
	where $\mathcal{Q}$ is partially orthogonal.
\end{definition}

We also introduce for the first time the notion of Gaussian random tensors. Our definition for Gaussian random tensors is motivated by the need to generate as few random numbers, or samples, as possible, while still being able to use the conclusions of literature on Gaussian random matrices. 

\begin{definition}[Gaussian random tensor]
	An $n_{1} \times n_{2} \times n_{3}$ tensor $\mathcal{W}$ is called a Gaussian random tensor, if the elements of $\mathcal{W}^{(1)}$ satisfy the standard normal distribution, and other frontal slices are all zeros. 
\end{definition}

The Fourier transform of $\mathcal{W}$ along the $3^\text{rd}$ dimension is denoted as $\hat{\mathcal{W}}$ such that every frontal slice is an identical copy of the first slice $\mathcal{W}^{(1)}$, see Figure~\ref{Omega}. 

\begin{figure}[h]\label{Omega}
\centering
\includegraphics[scale = .35]{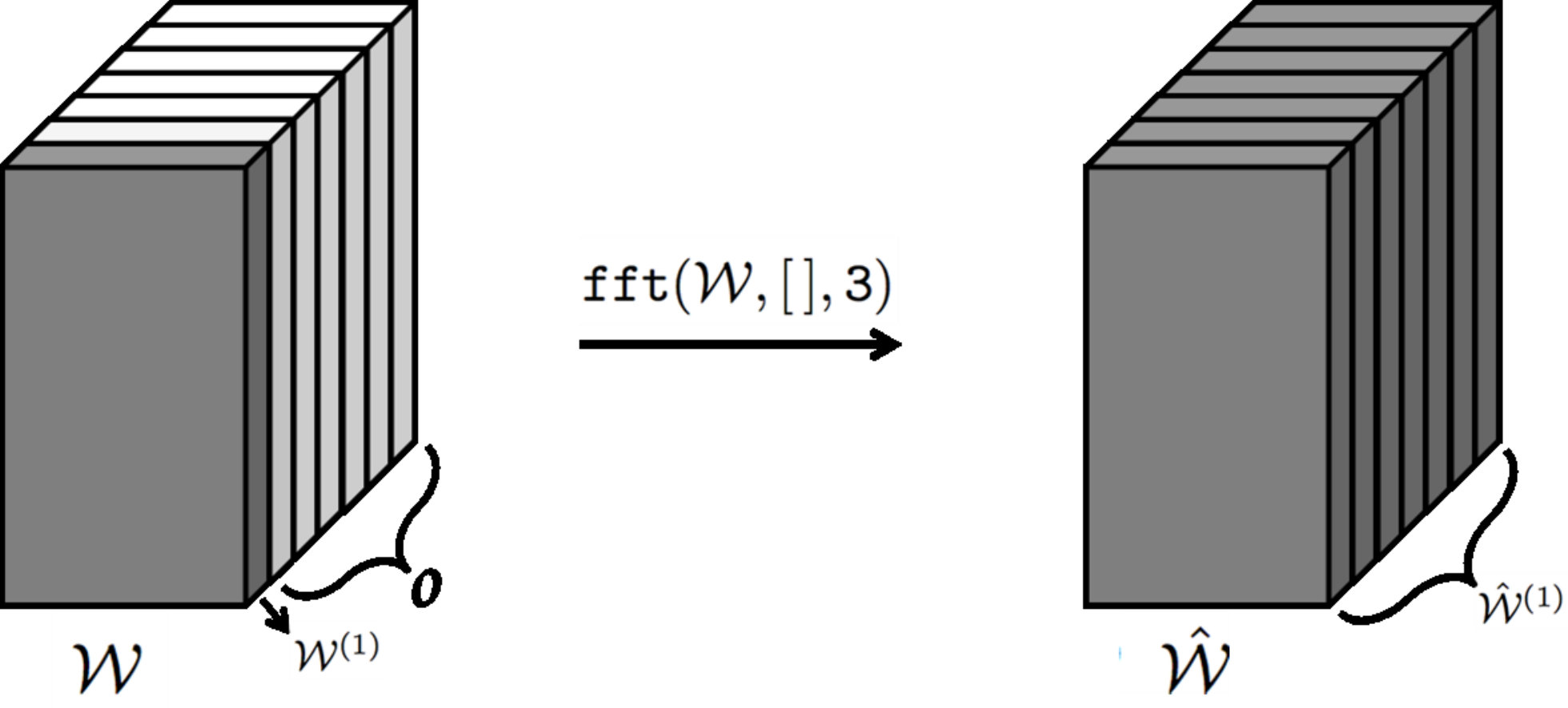}
\caption{A third order Gaussian random tensor and its Fourier transform.}
\end{figure}

\subsection{t-SVD}

We now present the t-SVD and truncated t-SVD, which builds on the operations of tensors introduced in Section~\ref{definitions}. These were first developed in~\cite{2011kilmer}.

\begin{definition}\cite{2011kilmer}
	Let $\mathcal{A}$ be an $n_{1} \times n_{2} \times n_{3}$ tensor. The t-SVD of $\mathcal{A}$ is
\begin{equation*}
\mathcal{A} = \mathcal{U} \ast \mathcal{S} \ast \mathcal{V}^{\rm T}
\end{equation*}                                                                 
where $\mathcal{U} \in \mathbb{R}^{n_{1} \times n_{1} \times n_{3}}$, $\mathcal{V} \in \mathbb{R}^{n_{2} \times n_{2} \times n_{3}}$ are partially orthogonal tensors, and $\mathcal{S}_{k} \in \mathbb{R}^{n_{1} \times n_{2} \times n_{3}}$ is a f-diagonal tensor.
\end{definition}

\begin{figure}[h]\label{tSVD}
\centering
\includegraphics[scale=.5]{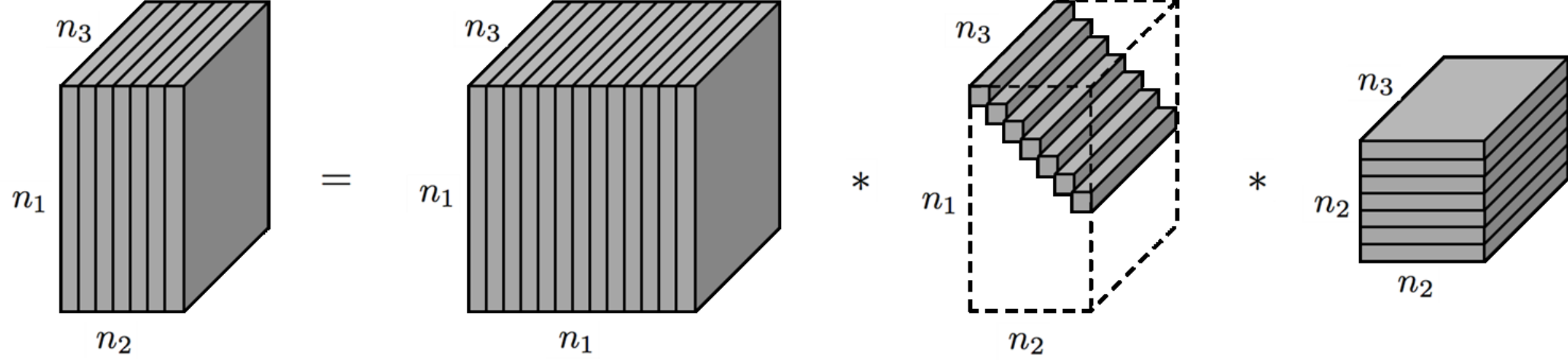}
\caption{The t-SVD of an $n_{1} \times n_{2} \times n_{3}$ tensor $\mathcal{A}$.}
\end{figure}

\begin{definition}\cite{2011kilmer}
Given a tensor $\mathcal{A} \in \mathbb{R}^{n_{1} \times n_{2} \times n_{3}}$, define the truncated t-SVD of $\mathcal{A}$ as
\begin{equation*}
\mathcal{A}_{k} = \mathcal{U}_{k} \ast \mathcal{S}_{k} \ast \mathcal{V}_{k}^{\rm T}
\end{equation*}                                                                 where $k$ is a target truncation term, $\mathcal{U}_{k} \in \mathbb{R}^{n_{1} \times k \times n_{3}}$, $\mathcal{V}_{k} \in \mathbb{R}^{n_{2} \times k \times n_{3}}$ are partially orthogonal, and $\mathcal{S}_{k} \in \mathbb{R}^{k \times k \times n_{3}}$ is a f-diagonal tensor. 
\end{definition}

The optimality of the error in the truncated t-SVD is presented below, which is a generalization of the well-known result of optimality of the truncated SVD~\cite[Section 7.4.2]{horn2012matrix}. 
Note, however, that truncation of the t-SVD to $k$ terms is not the same as computing a rank-k tensor approximation (i.e., a tensor approximation as the sum of $k$ outer products of vectors), nor is it equivalent to
a best rank-$(R_1,R_2,R_3)$ Tucker3 approximation.   What we show is that truncating the t-SVD gives the best ``tubal-rank k'' approximation~\cite{kilmer2013third}.  

Here, and henceforth, we will use the short-hand notation $\sum_{j>k}$ to represent $\sum_{j=k+1}^{\min\{n_1,n_2\}}$ for clarity. 
\begin{theorem}\label{theorem_theoretical_error}\cite{2011kilmer}
Given a tensor $\mathcal{A} \in \mathbb{R}^{n_{1} \times n_{2} \times n_{3}}$, $ \mathcal{A}_{k} = \operatorname{arg\,min}_{\tilde{\mathcal{A}} \in \mathcal{M}} \| \mathcal{A} - \tilde{\mathcal{A}} \|_{\rm F}$, where $\mathcal{M} = \lbrace \mathcal{C} = \mathcal{X} \ast \mathcal{Y} \mid \mathcal{X} \in \mathbb{R}^{n_{1} \times k \times n_{3}}, \mathcal{Y} \in \mathbb{R}^{k \times n_{2} \times n_{3}}\rbrace$. Therefore,  $\| \mathcal{A} - \mathcal{A}_{k} \|_{\rm F}$ is the theoretical minimal error, given by
\begin{equation} \label{e_ttsvd_error}
\| \mathcal{A} - \mathcal{A}_{k} \|_{\rm F} \> = \> \left(\frac{1}{n_3} \sum_{i=1}^{n_3}\sum_{j>k} (\hat{\sigma}_j^{(i)})^2\right)^{1/2} ,
\end{equation}
where $\hat{\sigma}_j^{(i)} \equiv \hat{\mc{S}}(j,j,i)$ is the $j^{th}$ singular value corresponding to the $i^{th}$ frontal face (in the Fourier domain).  
\end{theorem}

The relevance of the Fourier transform is perhaps not immediately obvious, until one examines how the truncated t-SVD is actually computed.  The procedure is given in Algorithm \ref{alg_truncated_t-svd}.

\begin{algorithm}[H]
\label{alg_truncated_t-svd}
\SetAlgoLined
\SetKwInOut{Input}{Input}\SetKwInOut{Output}{Output}
\Input{$\mathcal{A} \in \mathbb{R}^{n_{1} \times n_{2} \times n_{3}}$ and target truncation term $k$ }
\Output{$\mathcal{U}_{k} \in \mathbb{R}^{n_{1} \times k \times n_{3}}$, $\mathcal{S}_{k} \in \mathbb{R}^{k \times k \times n_{3}}$, and $\mathcal{V}_{k} \in \mathbb{R}^{n_{2} \times k \times n_{3}}$}
\BlankLine
 $\hat{\mathcal{A}} \leftarrow \tt{fft}(\mathcal{A},[\,],3)$\;

 \For{$i\leftarrow 1$ \KwTo $n_{3}$}{
 $[\mathbf{U},\mathbf{S},\mathbf{V}] = svd(\mathcal{A}^{(i)})$ \;
 Form $\mathbf{U}_{k}$, $\mathbf{S}_{k}$, and $\mathbf{V}_{k}$ by truncating $\mathbf{U}$, $\mathbf{S}$, and $\mathbf{V}$ with target truncation term $k$\;
 $\hat{\mathcal{U}}_{k}^{(i)} = \mathbf{U}_{k}$;
 $\hat{\mathcal{S}}_{k}^{(i)} = \mathbf{S}_{k}$; 
 $\hat{\mathcal{V}}_{k}^{(i)} = \mathbf{V}_{k}$;   
   }
  $\mathcal{U}_{k} \leftarrow \tt{ifft}(\hat{\mathcal{U}}_{k},[\,],3)$;
   $\mathcal{S}_{k} \leftarrow \tt{ifft}(\hat{\mathcal{S}}_{k},[\,],3)$;
   $\mathcal{V}_{k} \leftarrow \tt{ifft}(\hat{\mathcal{V}}_{k},[\,],3)$;
\caption{k-term truncated t-SVD~\cite{2011kilmer}}
\end{algorithm}

For a tensor $\mathcal{A} \in \mathbb{R}^{n_{1} \times n_{2} \times n_{3}}$, computing a k-term truncated t-SVD using Algorithm \ref{alg_truncated_t-svd} takes $\mc{O}(n_{1}n_{2}n_{3}k)$ flops. For a matrix $\mathbf{A} \in \mathbb{R}^{n_{1}n_{2} \times n_{3}}$, computing a k-term truncated SVD takes $\mc{O}(n_{1}n_{2}n_{3}k)$ flops as well. However, Algorithm \ref{alg_truncated_t-svd} can be computed in parallel over the frontal slices on a cluster, whereas typical algorithms used for the truncated SVD of a matrix cannot be computed in parallel.

\section{Randomized Tensor SVD}\label{rt-svd} 
In this section, we propose the rt-SVD method, which extends the matrix r-SVD method to the t-SVD. The goal of the rt-SVD (randomized tensor SVD) method is to find a good approximate factorization of tensor $\mathcal{A} \in \mathbb{R}^{n_{1} \times n_{2} \times n_{3}}$, $\mathcal{U}_{k} \ast \mathcal{S}_{k} \ast \mathcal{V}_{k}^{\rm T} $. There are two main steps which is summarized in Algorithm~\ref{alg_rt-svd}. The first step is to find a tensor $\mathcal{Q}$ such that 
\begin{equation*}
\mathcal{A}  \approx\mathcal{Q} \ast \mathcal{Q}^{\rm T} \ast \mathcal{A},
\end{equation*}
in a manner that will be made precise, 
and the second step is to connect this low-tubal-rank representation to a rt-SVD factorization. 

\begin{algorithm}[!ht]\label{alg_rt-svd}
\SetAlgoLined
\SetKwInOut{Input}{Input}\SetKwInOut{Output}{Output}
\Input{$\mathcal{A} \in \mathbb{R}^{n_{1} \times n_{2} \times n_{3}}$, target truncation term $k$, and oversampling parameter $p$}
\Output{$\mathcal{U}_{k} \in \mathbb{R}^{n_{1} \times k \times n_{3}}$, $\mathcal{S}_{k} \in \mathbb{R}^{k \times k \times n_{3}}$, and $\mathcal{V}_{k} \in \mathbb{R}^{n_{2} \times k \times n_{3}}$}
\BlankLine
 Generate a Gaussian random tensor $\mathcal{W} \in \mathbb{R}^{n_{2} \times (k+p) \times n_{3}}$\;
 Form a random projection of tensor $\mathcal{A}$ as $\mathcal{Y} = \mathcal{A} \ast \mathcal{W}$\;
Construct the tensor $\mathcal{Q}$ by using t-QR factorization\;
 Form a tensor $\mathcal{B}=\mathcal{Q}^{\rm T} \ast \mathcal{A}$, whose size is $(k+p) \times n_{2} \times n_{3}$\;
 Compute t-SVD of $\mathcal{B}$, truncate it with target truncation term $k$, and obtain $\mathcal{U}$, $\mathcal{S}_{k}$, and $\mathcal{V}_{k}$\;
 Form the rt-SVD of $\mathcal{A}$, $\mathcal{A} \approx (\mathcal{Q} \ast \mathcal{U}) \ast \mathcal{S}_{k} \ast \mathcal{V}^{\rm T}_{k} = \mathcal{U}_{k} \ast \mathcal{S}_{k} \ast \mathcal{V}^{\rm T}_{k}$.
\caption{rt-SVD, spatial domain version}
\end{algorithm}

For the convenience of error analysis , we present an implementation of Algorithm~\ref{alg_rt-svd} in the Fourier domain. This allows us to apply results from the matrix r-SVD independently to each frontal slice.

\begin{algorithm}[!ht]\label{alg_rt-svd_fourier}
\SetAlgoLined

\SetKwInOut{Input}{Input}\SetKwInOut{Output}{Output}
\Input{$\mathcal{A} \in \mathbb{R}^{n_{1} \times n_{2} \times n_{3}}$, target truncation term $k$, and parameter $p$ }
\Output{$\mathcal{U}_{k} \in \mathbb{R}^{n_{1} \times k \times n_{3}}$, $\mathcal{S}_{k} \in \mathbb{R}^{k \times k \times n_{3}}$, and $\mathcal{V}_{k} \in \mathbb{R}^{n_{2} \times k \times n_{3}}$}
\BlankLine
Generate a Gaussian random tensor $\mathcal{W} \in \mathbb{R}^{n_{2} \times (k+p) \times n_{3}}$\;
 $\hat{\mathcal{A}} \leftarrow \tt{fft}(\mathcal{A},[\,],3)$ and $\hat{\mathcal{W}} \leftarrow \tt{fft}(\mathcal{W},[\,],3)$\;
  \For{$i\leftarrow 1$ \KwTo $n_{3}$}{
 $\hat{\mathcal{Y}}^{(i)}=\hat{\mathcal{A}}^{(i)}\hat{\mathcal{W}}^{(i)}$ \;
  $[\hat{\mathcal{Q}}^{(i)},\hat{\mathcal{R}}^{(i)}]=\tt{qr}(\hat{\mathcal{Y}}^{(i)},0)$\;
  $\hat{\mathcal{B}}^{(i)}=(\hat{\mathcal{Q}}^{(i)})^{\rm H}\hat{\mathcal{A}}^{(i)}$\footnotemark \;
  $[\hat{\mathcal{U}}^{(i)},\hat{\mathcal{S}}^{(i)},\hat{\mathcal{V}}^{(i)}]=\tt{svd}(\hat{\mathcal{B}}^{(i)})$\;
  $\hat{\mathcal{U}}_{k}^{(i)}=\hat{\mathcal{Q}}^{(i)}\hat{\mathcal{U}}^{(i)}$; $\hat{\mathcal{S}}_{k}^{(i)}=\hat{\mathcal{S}}^{(i)}(1:k,1:k)$; $\hat{\mathcal{V}}_{k}^{(i)}=\hat{\mathcal{V}}^{(i)}(:,1:k)$.
   }
 $\mathcal{U}_{k} \leftarrow \tt{ifft}(\hat{\mathcal{U}}_{k},[\,],3)$;
   $\mathcal{S}_{k} \leftarrow \tt{ifft}(\hat{\mathcal{S}}_{k},[\,],3)$;
   $\mathcal{V}_{k} \leftarrow \tt{ifft}(\hat{\mathcal{V}}_{k},[\,],3)$.
\caption{rt-SVD, Fourier domain version}
\end{algorithm}

We now present a theorem that gives the expected error of $\|  \mathcal{A}-\mathcal{Q}*\mathcal{Q}^{\rm T} \ast \mathcal{A} \|_{\rm F} $ where the tensor $\mathcal{Q}$ is computed using Algorithm~\ref{alg_rt-svd}.  

\begin{theorem}\label{expectation}
Given an $n_{1} \times n_{2} \times n_{3}$ tensor $\mathcal{A}$  and an $n_{2} \times (k+p) \times n_{3}$ Gaussian random tensor $\mathcal{W}$, if $\mathcal{Q}$ is obtained from t-QR of $\mathcal{Y}  =\mathcal{A} \ast \mathcal{W}$, then
\begin{equation*}
\mathbb{E} \|  \mathcal{A}-\mathcal{Q}*\mathcal{Q}^{\rm T} \ast \mathcal{A} \|_{\rm F} 
 \leq  \sqrt{1+\dfrac{k}{p-1}} \left(\frac{1}{n_3}\sum_{i=1}^{n_{3}}\sum_{j>k} (\hat{\sigma}^{(i)}_{j})^{2}\right)^{1/2}.
\end{equation*}
where $k$ is a target truncation term, $p \geq 2$ is the oversampling parameter, and $\hat{\sigma}_{j}^{(i)}$ is the $i_{th}$ component of $\tt{fft}(\mathcal{S}(j,j,:),[\,],3)$.
\end{theorem}
\begin{proof}
See Appendix~\ref{app_tensor}.
\end{proof}

\footnotetext{Since slices of $\hat{\mathcal{A}}$ can be complex in general, we use the notation superscript $\rm H$ instead of  superscript $\rm T$ here.}

Theorem~\ref{expectation} is important because it shows that, in expectation, the error in the rt-SVD algorithm is within a factor $\sqrt{1+\dfrac{k}{p-1}}$ of the optimal result in Theorem~\ref{theorem_theoretical_error}. Note that this is the same optimality factor that one obtains in the matrix case, see Theorem~\ref{expected_error}.

Theorem~\ref{expectation} also shows how the error of the low-tubal-rank approximation  $ \mathcal{Q} \ast \mathcal{Q}^{\rm T} \ast \mathcal{A} $ relies on the decay of singular values of the frontal slices $\hat{\mathcal{A}}^{(i)}$. If the singular values decay rapidly, we can approximate $(\sum_{j>k} \hat{\sigma}^{(i)}_{j})^{1/2}$ by $\hat{\sigma}^{(i)}_{k+1}$  and therefore 
\begin{equation*}
\mathbb{E}\, \|  \mathcal{A}-\mathcal{Q} \ast \mathcal{Q}^{\rm T} \ast \mathcal{A} \|_{\rm F} \leq  \sqrt{ 1+\dfrac{k}{p-1}} \,\max_{1\leq i\leq n_3}\hat{\sigma}^{(i)}_{k+1}. 
\end{equation*}

If the singular values of $\hat{\mathcal{S}}^{(i)}$ decay gradually, we can instead use the following approximation $\sum_{j>k}(\hat{\sigma}^{(i)}_{j})^2 \leq (m-k) (\hat{\sigma}^{(i)}_{k+1})^2$ where $\hat{\sigma}^{(i)}_{k+1}$ is the $(k+1)^{th}$ singular value of the $i^{th}$ frontal slice in the Fourier domain, and $m = \min \{n_1,n_2\}$. However,  $\hat{\sigma}^{(i)}_{k+1}$ can be large, in which case the accuracy of the rt-SVD may be poor. We present a new algorithm (Algorithm~\ref{alg_subspace}) for tensor low-tubal-rank representation based on the t-product that applies the randomized subspace iteration to each frontal slice in the Fourier domain. 

\begin{algorithm}[!ht]\label{alg_subspace}
\SetAlgoLined

\SetKwInOut{Input}{Input}\SetKwInOut{Output}{Output}
\Input{$\mathcal{A} \in \mathbb{R}^{n_{1} \times n_{2} \times n_{3}}$, target truncation term $k$, oversampling parameter $p$, the number of iterations q }
\Output{$\mathcal{U}_{k} \in \mathbb{R}^{n_{1} \times k \times n_{3}}$, $\mathcal{S}_{k} \in \mathbb{R}^{k \times k \times n_{3}}$, and $\mathcal{V}_{k} \in \mathbb{R}^{n_{2} \times k \times n_{3}}$}
\BlankLine

 Generate a Gaussian random tensor $\mathcal{W} \in \mathbb{R}^{n_{2} \times (k+p) \times n_{3}}$\;
 Form a tensor $\mathcal{Y}_{0} = \mathcal{A} \ast \mathcal{W}$ and compute the t-QR factorization $\mathcal{Y}_{0} = \mathcal{Q}_{0} \ast \mathcal{R}_{0}$ \;
  \For{$i\leftarrow 1$ \KwTo $q$}{
 $\mathcal{\tilde{Y}}_{i} = \mathcal{A}^{\tt{T}} \ast \mathcal{Q}_{i-1}$ and compute the t-QR factorization $\mathcal{\tilde{Y}}_{i} = \mathcal{\tilde{Q}}_{i} \ast \mathcal{\tilde{R}}_{i}$\;
   $\mathcal{Y}_{i} = \mathcal{A} \ast \mathcal{\tilde{Q}}_{i}$ and compute the t-QR factorization $\mathcal{Y}_{i} = \mathcal{Q}_{i} \ast \mathcal{R}_{i}$\;
   }
Form a tensor $\mathcal{Q} = \mathcal{Q}_{q} $\;
 Form a tensor $\mathcal{B}=\mathcal{Q}^{\rm T} \ast \mathcal{A}$, the size of $\mathcal{B}$ is $(k+p) \times n_{2} \times n_{3}$ which is smaller than tensor $\mathcal{A}$\;
 Compute t-SVD of $\mathcal{B}$, truncate it, and obtain $\mathcal{U}$, $\mathcal{S}_{k}$,  $\mathcal{V}_{k}$\;
 Form the rt-SVD of $\mathcal{A}$, $\mathcal{A} \approx (\mathcal{Q} \ast \mathcal{U}) \ast \mathcal{S}_{k} \ast \mathcal{V}^{\rm T}_{k} = \mathcal{U}_{k} \ast \mathcal{S}_{k} \ast \mathcal{V}^{\rm T}_{k}$.
\caption{rt-SVD with subspace iteration, spatial domain version}
\end{algorithm}

Algorithm~\ref{alg_subspace} works efficiently when the singular values, given by the diagonals of $\hat{\mathcal{S}}^{(i)}$, decay at the same rate across each frontal slice of $\widehat{\mathcal{A}}$. However, when the singular values decay gradually only for some slices, Algorithm \ref{alg_subspace} may be wasteful in terms of computational costs, since some iterations can be stopped earlier than others. This can be avoided if a different number of iterations $q_i$ is used for each frontal slice.  Let us define the iteration vector as $\mathbf{q}=(q_{1}, q_{2},\dots,q_{n_{3}})^\top$. We present an algorithm (Algorithm~\ref{alg_subspace_fourier}) that employs different iteration count in each frontal slice.

\begin{algorithm}[!ht]\label{alg_subspace_fourier}
\SetAlgoLined

\SetKwInOut{Input}{Input}\SetKwInOut{Output}{Output}
\Input{$\mathcal{A} \in \mathbb{R}^{n_{1} \times n_{2} \times n_{3}}$, target truncation term $k$, parameter $p$, and the iterations vector $\mathbf{q}$ }
\Output{$\mathcal{U}_{k} \in \mathbb{R}^{n_{1} \times k \times n_{3}}$, $\mathcal{S}_{k} \in \mathbb{R}^{k \times k \times n_{3}}$, and $\mathcal{V}_{k} \in \mathbb{R}^{n_{2} \times k \times n_{3}}$}
\BlankLine

 Generate a Gaussian random tensor $\mathcal{W} \in \mathbb{R}^{n_{2} \times (k+p) \times n_{3}}$\;

 $\hat{\mathcal{A}} \leftarrow \tt{fft}(\mathcal{A},[\,],3)$ and $\hat{\mathcal{W}} \leftarrow \tt{fft}(\mathcal{W},[\,],3)$\;

  \For{$i\leftarrow 1$ \KwTo $n_{3}$}{
  $\hat{\mathcal{Y}}^{(i)}=\hat{\mathcal{A}}^{(i)}\hat{\mathcal{W}}^{(i)}$ \;
 
  $[\hat{\mathcal{Q}}_{j-1}^{(i)},\sim]=\tt{qr}(\hat{\mathcal{Y}}^{(i)},0)$\;
   \For{$j\leftarrow 1$ \KwTo $q_{i}$}{
   $\hat{\mathcal{Z}}_{j}^{(i)}=(\hat{\mathcal{A}}^{(i)})^{\rm H}\hat{\mathcal{Q}}_{j-1}^{(i)}$\;
   $[\hat{\mathcal{G}}_{j}^{(i)},\sim]=\tt{qr}(\hat{\mathcal{Z}}^{(i)},0)$\;
   $\hat{\mathcal{Y}}_{j}^{(i)}=\hat{\mathcal{A}}^{(i)}\hat{\mathcal{G}}_{j}^{(i)}$ \;
   $[\hat{\mathcal{Q}}_{j}^{(i)},\sim]=\tt{qr}(\hat{\mathcal{Y}}^{(i)},0)$\;
   }
  Form $\hat{\mathcal{Q}}^{(i)}$ as $\hat{\mathcal{Q}}^{(i)}=\hat{\mathcal{Q}}_{j}^{(i)}$\;
  $\hat{\mathcal{B}}^{(i)}=(\hat{\mathcal{Q}}^{(i)})^{\rm H}\hat{\mathcal{A}}^{(i)}$\;
  
  $[\hat{\mathcal{U}}^{(i)},\hat{\mathcal{S}}^{(i)},\hat{\mathcal{V}}^{(i)}]=\tt{svd}(\hat{\mathcal{B}}^{(i)})$\;
  
  $\hat{\mathcal{U}}_{k}^{(i)}=\hat{\mathcal{Q}}^{(i)}\hat{\mathcal{U}}^{(i)}$; $\hat{\mathcal{S}}_{k}^{(i)}=\hat{\mathcal{S}}^{(i)}(1:k,1:k)$; $\hat{\mathcal{V}}_{k}^{(i)}=\hat{\mathcal{V}}^{(i)}(:,1:k)$.
   }
  $\mathcal{U}_{k} \leftarrow \tt{ifft}(\hat{\mathcal{U}}_{k},[\,],3)$;
   $\mathcal{S}_{k} \leftarrow \tt{ifft}(\hat{\mathcal{S}}_{k},[\,],3)$;
   $\mathcal{V}_{k} \leftarrow \tt{ifft}(\hat{\mathcal{V}}_{k},[\,],3)$.
\caption{rt-SVD with subspace iterations, Fourier domain version} \end{algorithm}

The expected error of the probabilistic part of Algorithm \ref{alg_subspace_fourier} is given in Theorem~\ref{expectation_subspace}.

\begin{theorem}\label{expectation_subspace}
Given an $n_{1} \times n_{2} \times n_{3}$ tensor $\mathcal{A}$  and an $n_{2} \times (k+p) \times n_{3}$ tensor $\mathcal{W}$, if $\mathcal{Q}$ is obtained from Algorithm \ref{alg_subspace_fourier}, then
\begin{equation*}
\mathbb{E}\, \normf{ \mathcal{A}-\mathcal{Q}*\mathcal{Q}^{\rm T} \ast \mathcal{A} } 
 \> \leq \>  
\left( \frac{1}{n_3}\sum^{n_{3}}_{i=1} \left(1 + \frac{k}{p-1}(\tau^{(i)}_k)^{4q_i}\right) \left(\sum_{j>k} (\hat{\sigma}^{(i)}_{j})^2\right) \right)^{1/2},
\end{equation*}
where $k$ is a target truncation term, $p \geq 2$ is the oversampling parameter, $\mathbf{q}$ is the iterations count vector, $\hat{\sigma}_{j}^{(i)}$ is the $i^{th}$ component of $\tt{fft}(\mathcal{S}(j,j,:),[\,],3)$, 
and the singular value gap $\hat{\tau}^{(i)}_{k}=\frac{\hat{\sigma}^{(i)}_{k+1}}{\hat{\sigma}^{(i)}_{j}} \ll 1$.
\end{theorem}
\begin{proof} See Appendix~\ref{app_tensor}.
\end{proof}

If the iteration count $q_i = q$ for all $i=1,\dots,n_3$, then we can use the following simpler bound 
\begin{equation}
\mathbb{E} \normf{  \mathcal{A}-\mathcal{Q}*\mathcal{Q}^{\rm T} \ast \mathcal{A} } \> 
 \leq  \> \sqrt{1+\dfrac{k}{p-1}  (\tau_k^{\max})^{4q} }\left(\frac{1}{n_3}\sum_{i=1}^{n_{3}}\sum_{j>k} (\hat{\sigma}^{(i)}_{j})^{2}\right)^{1/2},
\end{equation}
where $\tau_k^{\max} = \max_{1\leq i \leq n_3} \tau_k^{(i)}$ is the largest singular value gap. In particular, $q=0$ gives the same result as Theorem~\ref{expectation}.

Theorem~\ref{expectation_subspace} suggests an effective strategy to pick the iteration count $q_i$. Suppose we are given a tolerance parameter $0 < \epsilon < 1$. Then we choose 
\begin{equation}
q_i = \left\lceil \frac{1}{4} \left. \log \frac{\epsilon (p-1)}{k} \middle/ \log \tau_k^{(i)} \right.\right\rceil,
\end{equation}
which ensures that the error in $\mathbb{E}\, \normf{ \mathcal{A}-\mathcal{Q}*\mathcal{Q}^{\rm T} \ast \mathcal{A} } $ is at most $\sqrt{1+\epsilon}$ of the optimal result in Theorem~\ref{theorem_theoretical_error}.

Theorems~\ref{expectation} and~\ref{expectation_subspace} provide insight into the average behavior of the error. This next result provides the tail bounds of the probabilistic error. 
\begin{theorem}\label{concentration}
With the assumptions of Theorem~\ref{expectation_subspace}, let $0 < \delta < 1$ be the failure probability and define the constant
\[ C_\delta = \frac{e\sqrt{k+p}}{p+1} \left(\frac{2}{\delta}\right)^{\frac{1}{p+1}} \left(\sqrt{n_2-k} + \sqrt{k+p} + \sqrt{2\log \frac{2}{\delta}}\right).\]
Then with  probability at most  $\delta$,
\begin{equation*}
\normf{ \mathcal{A}-\mathcal{Q}*\mathcal{Q}^{\rm T} \ast \mathcal{A} }^2 
 \> \leq \>  
 \frac{1}{n_3}\sum^{n_{3}}_{i=1} \left(1 + C_\delta^2(\tau^{(i)}_k)^{4q_i}\right) \left(\sum_{j>k} (\hat{\sigma}^{(i)}_{j})^2\right) .
\end{equation*}
 
\end{theorem}

\begin{proof}
See Appendix~\ref{app_tensor}.
\end{proof}

Theorem~\ref{concentration} shows that although the result holds with high probability, there is an arbitrary small chance that the upper bound may not hold. For the sample values $n_2 = 100$, $k = 30$, $p = 20$ and $\delta = 10^{-16}$, we obtain $C_\delta \approx 43$. 

We now discuss the computational cost of Algorithm~\ref{alg_subspace_fourier}. Recall that the t-SVD requires $\mc{O}\left(n_{1}n_{2}n_{3}\log n_{3}\right)$ flops to transform to the Fourier domain and an additional $\mc{O}\left(n_{1}n_{2}n_{3}\min\{n_1,n_2\}\right)$ flops for the decomposition. 
On the other hand, the rt-SVD method only requires $\mc{O}\left(kn_{1}n_{2}n_{3}\right)$. The rt-SVD can be advantageous when $k \ll \min\{ n_1,n_2\}$.

\section{Numerical Results}\label{Numerical}
In this section, we provide some numerical results on the accuracy of the proposed low-rank representations, as well as on the computation time of the proposed methods. The proposed algorithms are demonstrated on an application to facial recognition. The datasets for the experiments are a subset of the Cropped Extended Yale Face Dataset B~\cite{CroppedYaleB} (abbreviated as Cropped Yale B dataset) and the dataset of faces maintained at AT$\&$T Laboratories Cambridge~\cite{Att} (abbreviated as AT$\&$T dataset). The Cropped Yale B dataset has $1140$ images that contains the first 
$30$ possible illuminations of  $38$ different people. Each image has $192 \times 168$ pixels in a grayscale range. This is collected into a $192 \times 1140 \times 168$ tensor, and this tensor is denoted by $\mathcal{B}$. The AT$\&$T dataset has $400$ images that contains $10$ different poses of  $40$ people. Each image has $112 \times 92$ pixels in a grayscale which is collected into a $112 \times 400 \times 92$ tensor, denoted as $\mathcal{E}$. The experiments were run on a laptop with $2.3$ GHz Intel Core i7 and  $8$ GB memory.
\subsection{Error Analysis}

In Section~\ref{rt-svd}, we derived theoretical results for the expected approximation errors of rt-SVD and rt-SVD with subspace iteration. Here, we provide some numerical results to demonstrate their comparative performance.  We compare the relative errors obtained by using rt-SVD, rt-SVD with subspace iteration and the relative theoretically minimal errors on the dataset $\mathcal{B}$. The target truncation term $k$ is allowed to vary between $50$ and $180$.
 We define the relative errors obtained by using the rt-SVD with subspace iteration as $e^{q}_{k}$,   
 \begin{equation}
 e^{q}_{k} \> = \> \frac{\| (\mc{I} - \mc{Q}\mc{Q}^\top)\mathcal{A}  \|_{\rm F}}{\| \mathcal{A}\|_{F}}, 
  \end{equation}
 where $q$ represents the number of iterations (see Algorithm~\ref{alg_subspace}) and $k$ denotes the target truncation term. Because the rt-SVD is a specific case of rt-SVD with subspace iteration with $q = 0$, we will use $e^{0}_{k}$ to denote the relative errors obtained by using rt-SVD.

 Theorem~\ref{theorem_theoretical_error}, gives us the best possible relative error $e_{k}$, as a function of the  target truncation term $k$.
 
  \begin{equation}
  	e_{k} \> \equiv \> \frac{\| \mathcal{A} - \mathcal{A}_{k} \|_{\rm F}}{\| \mathcal{A}\|_{F}} = 
	\frac{\| \hat{\mathcal{S}}(k+1:n,k+1:n,:) \|_{\rm F}}{\| \hat{\mathcal{S}}\|_{F}}
  \end{equation}
  
\begin{figure}[h]\label{theoretical_error}
\centering
\includegraphics[scale=.35]{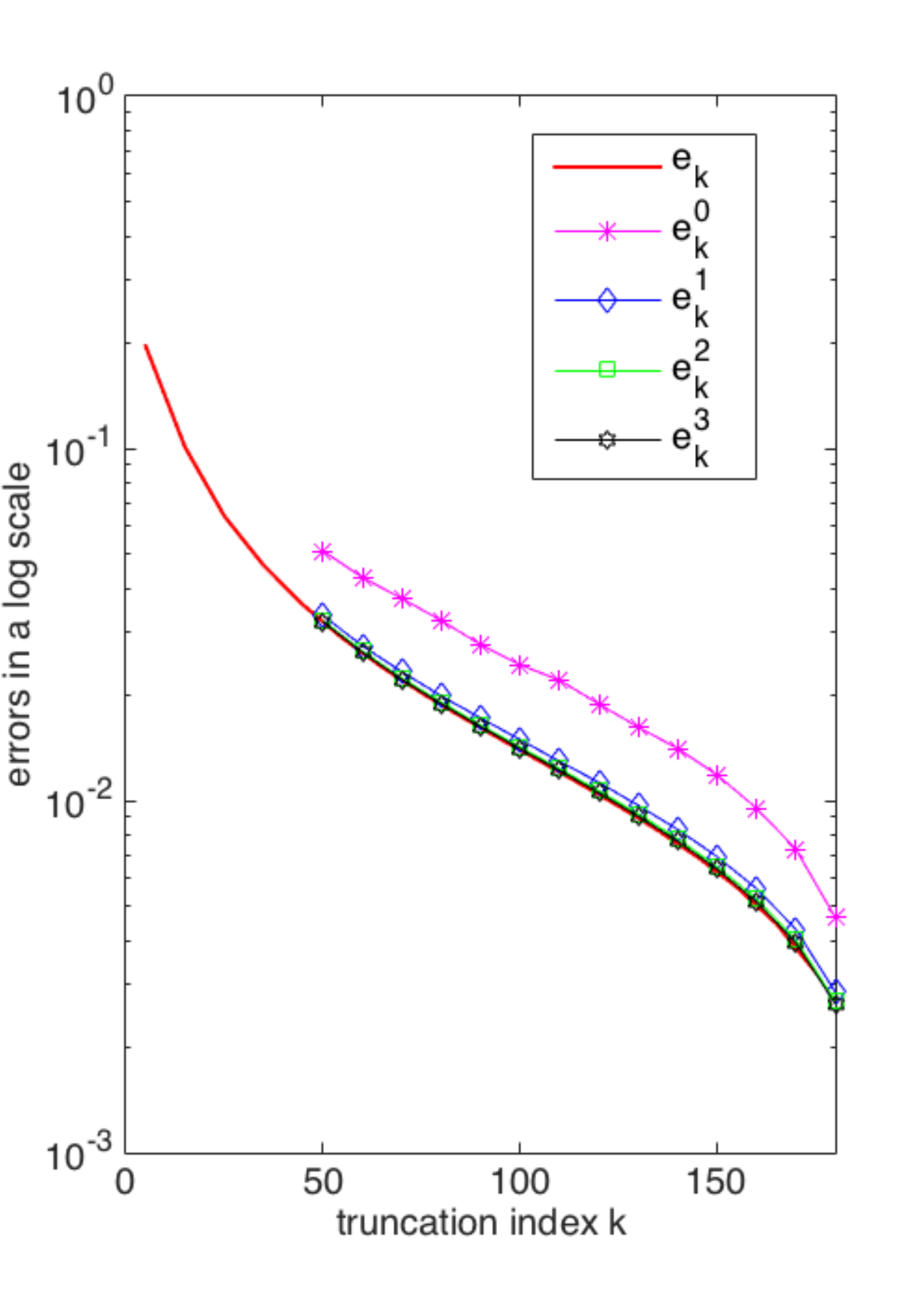}
\includegraphics[scale=.35]{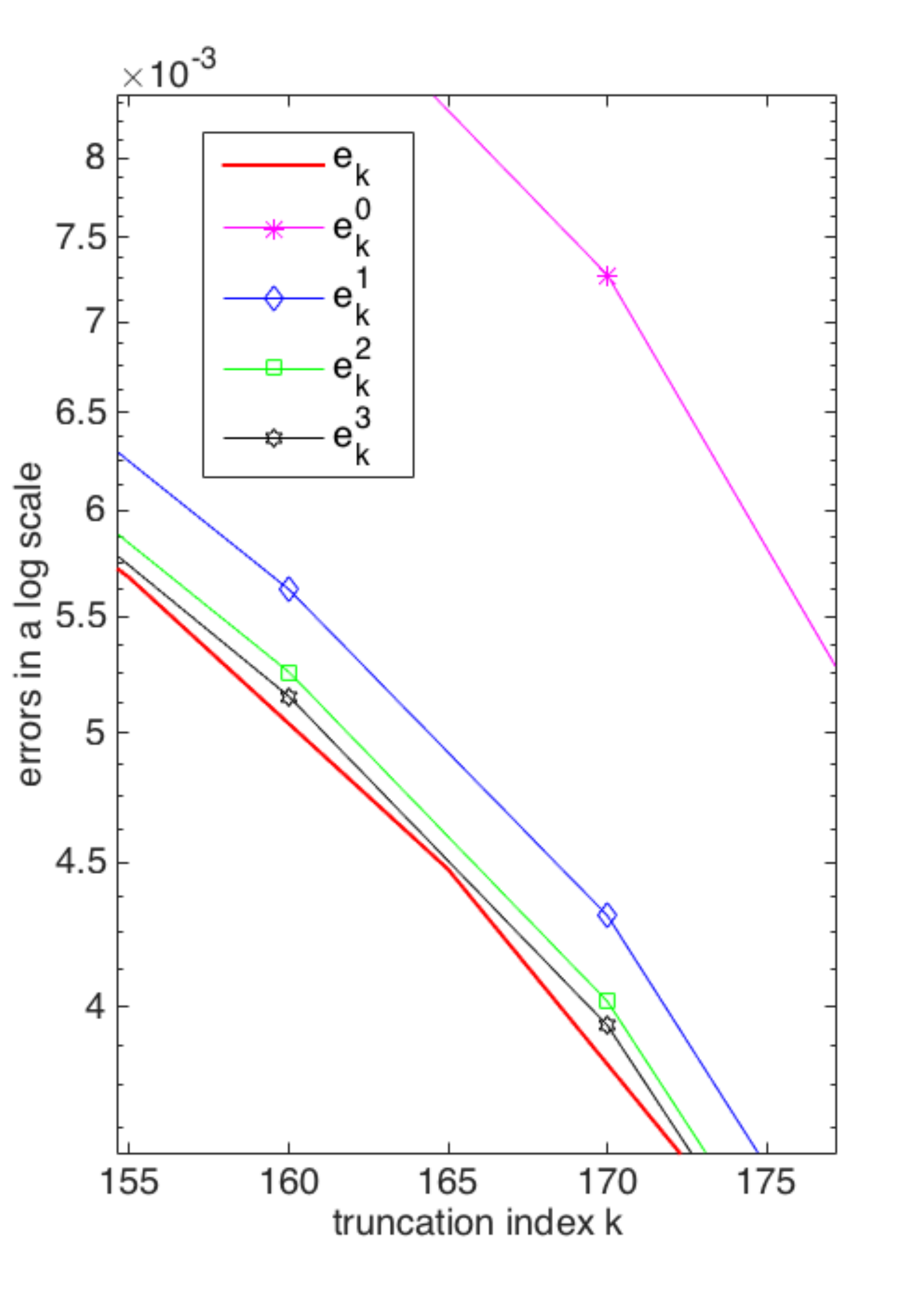}
\caption{(left) The comparison of the theoretically minimal errors and the errors of rt-SVD with subspace iterations. (right) The zoomed in version of the left panel.}
\end{figure}
Figure~\ref{theoretical_error} shows that the errors $e^{q}_{k}$, with different number of iterations, have the similar convergence trajectories and are quite close to the best possible theoretical error $e_{k}$.  In other words, rt-SVD and rt-SVD with subspace iteration are both comparable in accuracy with the truncated t-SVD. Moreover, Figure~\ref{theoretical_error} shows that $e^{q}_{k}$ approaches $e_{k}$ when the number of iterations $q$ increases, yielding a more accurate approximation. 

\subsection{Facial Recognition}
In this section,  we apply the rt-SVD and rt-SVD with subspace iterations on the Cropped Yale B dataset and AT$\&$T dataset. The images from each database is split into a training and test datasets. The tensors are constructed in such a way that the various images stored as lateral slices of the tensor. We process the training dataset and store a projector tensor $\mathcal{U}_{k}$ and a coefficient tensor $\mathcal{C}$ of smaller dimensions as described in Table~\ref{procedure}.  For each image in the test dataset, we obtain a tensor coefficient by projecting onto the training dataset, and the face is recognized as the lateral slice with the closest distance to the tensor coefficient.  The procedure for processing the training  and test datasets is shown in Table~\ref{procedure}.

\begin{table}[!ht]
\caption{The procedure of facial recognition based on t-SVD method}\label{procedure}
\begin{center}
\begin{tabular}{|p{7cm}|p{7cm}|}
\hline
\multicolumn{2}{|c|}{Facial Recognition Procedure}\\
\hline
For the training Dataset: & For the new image in the test dataset:\\
\hline
1. Form the training dataset into a third order tensor and calculate the mean lateral slice across the second dimension; & 1. Form the new image as a tensor with only one lateral slice and subtract the mean lateral slice from it; \\
& \\
2. Calculate standard mean-shifted tensor and denote it as tensor $\mathcal{A}$; & 2. Compute the standard mean-shifted lateral slice and denote it as $\mathcal{T}$;\\
& \\
3. Compute the truncated t-SVD of $\mathcal{A}$ or the approximated truncated t-SVD of $\mathcal{A}$ with target truncation term $k$; & 3. Compute the coefficient tensor $\mathcal{C}_{t}=\mathcal{U}_{k}^{\tt{T}} \ast \mathcal{T}$;\\
& \\
4. Compute the coefficient tensor $\mathcal{C}=\mathcal{U}_{k}^{\tt{T}} \ast \mathcal{A}$, and store it with projector tensor $\mathcal{U}_{k}$. & 4. Find the smallest distance of $\mathcal{C}_{t}$ with each lateral slices of  $\mathcal{C}$.\\
\hline 

\end{tabular}
\end{center}
\end{table}

To measure the performance more rigorously, we use 10-fold cross-validation. In 10-fold cross-validation, the dataset is randomly partitioned into $10$ equal-size subsets. In the $k^{th}$ trial, the $k$ subset (also referred to as a fold) is used as the test dataset, whereas the other $9$ subsets are simultaneously used as training dataset. Therefore, the algorithm will be tested $10$ times with 10 different combinations of the same dataset. The randomized algorithms are run $20$ times for each fold to compute the mean, maximum, and minimum of recognition rates. The recognition rate here is defined as 
\begin{equation*}
	r = \frac{\tt{the\, number \,of \,images\, recognized \,correctly}}{\tt{the\, number\, of \,test \,images}}. 
\end{equation*}

\subsubsection{Cropped Yale Face B Dataset}\label{YaleB}
There are 1140 images in Cropped Yale B Dataset, so the size of training dataset is  $192 \times 1026 \times 168$ and the size of test dataset is  $192 \times 114 \times 168$ in each fold. A few sample images from the Cropped Yale B dataset under different illuminations are shown in Figure~\ref{CPpeople}. 

\begin{figure}[!ht]\label{CPpeople}
\centering
\includegraphics[scale = .7]{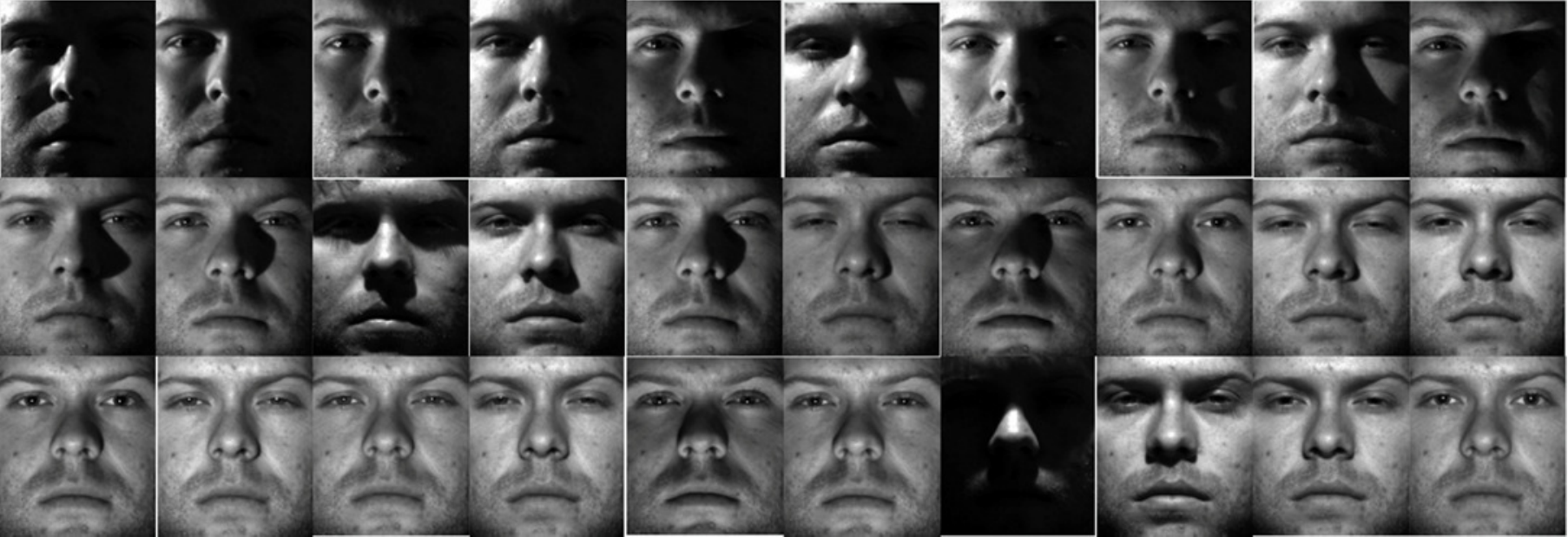}
\caption{Sample images from Cropped Yale B Dataset}.
\end{figure}

Table \ref{rank25_CP} shows the accuracy of the rt-SVD algorithm, whereas Table~\ref{rank50_CP} shows the accuracy of rt-SVD with subspace iterations. The computational times are reported in  Figure~\ref{CP_run}. We observe that as the truncation term gets larger, the  recognition rate increases, as well as the  computation time. In practice, the target truncation term $k$ depends on the tradeoff between recognition rate and computational time. We report results for target truncation term ranging from $25$ to $50$.

\begin{table}[!ht]
\centering
\caption{Recognition Rates on Cropped Yale B dataset with $k= 25$}
\begin{tabular}{|l|l|l|l|l|l|l|l|l|l|l|}
\hline
r & fold 1 & fold 2 & fold 3 & fold 4 & fold 5 & fold 6 & fold 7 & fold 8 & fold 9 & fold 10 \\ \hline
 \multicolumn{11}{|c|}{The t-SVD method}\\
\hline
&0.9912 & 1.0000 & 0.9211 & 1.0000 & 1.0000 & 0.9912 & 0.9035 & 0.9737 & 0.7368 & 0.9825 \\
   \hline
 \multicolumn{11}{|c|}{The rt-SVD method}\\
\hline
mean &0.9912 & 1.0000 & 0.9175 & 0.9943 & 1.0000 & 0.9912 & 0.9035 & 0.9737 & 0.7368 & 0.9772 \\

min &0.9912 & 1.0000 & 0.9123 & 0.9912 & 1.0000 & 0.9912 & 0.9035 & 0.9737 & 0.7368 & 0.9737 \\
max & 0.9912 & 1.0000 & 0.9211 & 1.0000 & 1.0000 & 0.9912 & 0.9035 & 0.9737 & 0.7368 & 0.9912 \\
 \hline
 \multicolumn{11}{|c|}{The rt-SVD method with subspace iterations $q=1$}\\
\hline
mean & 0.9912 & 1.0000 & 0.9211 & 1.0000 & 1.0000 & 0.9912 & 0.9035 & 0.9737 & 0.7368 & 0.9833 \\
min &0.9912 & 1.0000 & 0.9211 & 1.0000 & 1.0000 & 0.9912 & 0.9035 & 0.9737 & 0.7368 & 0.9737 \\
max & 0.9912 & 1.0000 & 0.9211 & 1.0000 & 1.0000 & 0.9912 & 0.9035 & 0.9737 & 0.7368 & 0.9912 \\
 \hline
 \multicolumn{11}{|c|}{The rt-SVD method with subspace iterations $q=2$}\\
\hline
mean & 0.9912 & 1.0000 & 0.9211 & 1.0000 & 1.0000 & 0.9912 & 0.9035 & 0.9737 & 0.7368 & 0.9882 \\
min & 0.9912 & 1.0000 & 0.9211 & 1.0000 & 1.0000 & 0.9912 & 0.9035 & 0.9737 & 0.7368 & 0.9825 \\
max & 0.9912 & 1.0000 & 0.9211 & 1.0000 & 1.0000 & 0.9912 & 0.9035 & 0.9737 & 0.7368 & 0.9912 \\
\hline

\end{tabular}\label{rank25_CP}
\end{table}

\begin{table}[!ht]
\centering
\caption{Recognition Rates on Cropped Yale B dataset with $k=50$}
\begin{tabular}{|l|l|l|l|l|l|l|l|l|l|l|}
\hline
r & fold 1 & fold 2 & fold 3 & fold 4 & fold 5 & fold 6 & fold 7 & fold 8 & fold 9 & fold 10 \\ \hline
\multicolumn{11}{|c|}{The t-SVD method}\\
\hline
 & 0.9912 & 1.0000 & 0.9298 & 1.0000 & 1.0000 & 0.9912 & 0.9035 & 0.9825 & 0.7368 & 0.9912 \\
 \hline
 \multicolumn{11}{|c|}{The rt-SVD method}\\
\hline
mean & 0.9912 & 1.0000 & 0.9298 & 1.0000 & 1.0000 & 0.9912 & 0.9035 & 0.9825 & 0.7368 & 0.9912 \\
min & 0.9912 & 1.0000 & 0.9298 & 1.0000 & 1.0000 & 0.9912 & 0.9035 & 0.9825 & 0.7368 & 0.9912 \\
max & 0.9912 & 1.0000 & 0.9298 & 1.0000 & 1.0000 & 0.9912 & 0.9035 & 0.9825 & 0.7368 & 0.9912 \\

 \hline
 \multicolumn{11}{|c|}{The rt-SVD method with subspace iterations $q=1$}\\
\hline
mean & 0.9912 & 1.0000 & 0.9298 & 1.0000 & 1.0000 & 0.9912 & 0.9035 & 0.9825 & 0.7368 & 0.9912 \\
min & 0.9912 & 1.0000 & 0.9298 & 1.0000 & 1.0000 & 0.9912 & 0.9035 & 0.9825 & 0.7368 & 0.9912 \\
max & 0.9912 & 1.0000 & 0.9298 & 1.0000 & 1.0000 & 0.9912 & 0.9035 & 0.9825 & 0.7368 & 0.9912 \\
\hline
 \multicolumn{11}{|c|}{The rt-SVD method with subspace iterations $q=2$}\\
\hline
mean &  0.9912 & 1.0000 & 0.9298 & 1.0000 & 1.0000 & 0.9912 & 0.9035 & 0.9825 & 0.7368 & 0.9912 \\
min &  0.9912 & 1.0000 & 0.9298 & 1.0000 & 1.0000 & 0.9912 & 0.9035 & 0.9825 & 0.7368 & 0.9912 \\
max & 0.9912 & 1.0000 & 0.9298 & 1.0000 & 1.0000 & 0.9912 & 0.9035 & 0.9825 & 0.7368 & 0.9912 \\
 \hline

\end{tabular}\label{rank50_CP}
\end{table}

\begin{figure}[!ht]
\centering
\includegraphics[scale=.5]{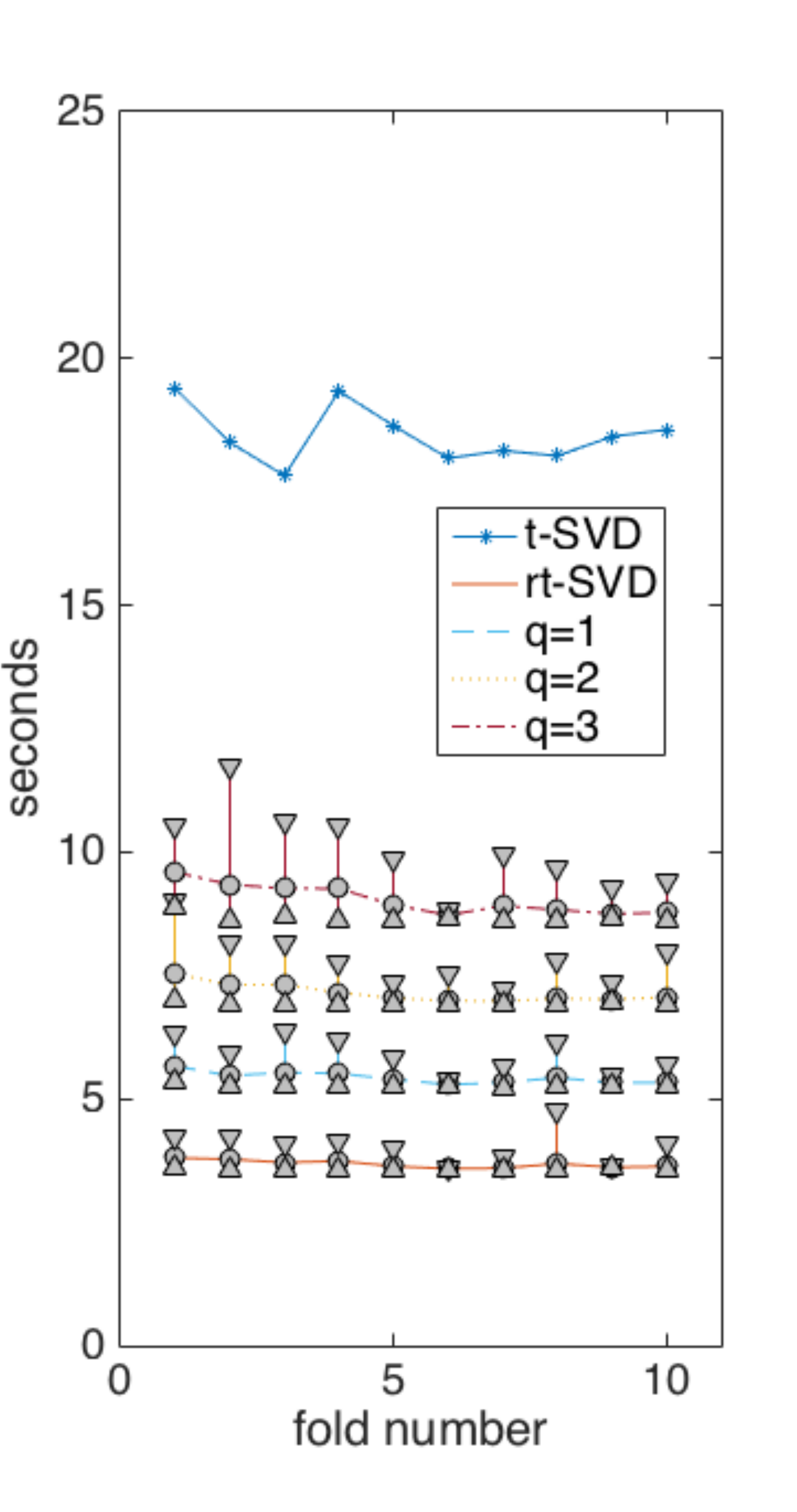}
\includegraphics[scale=.5]{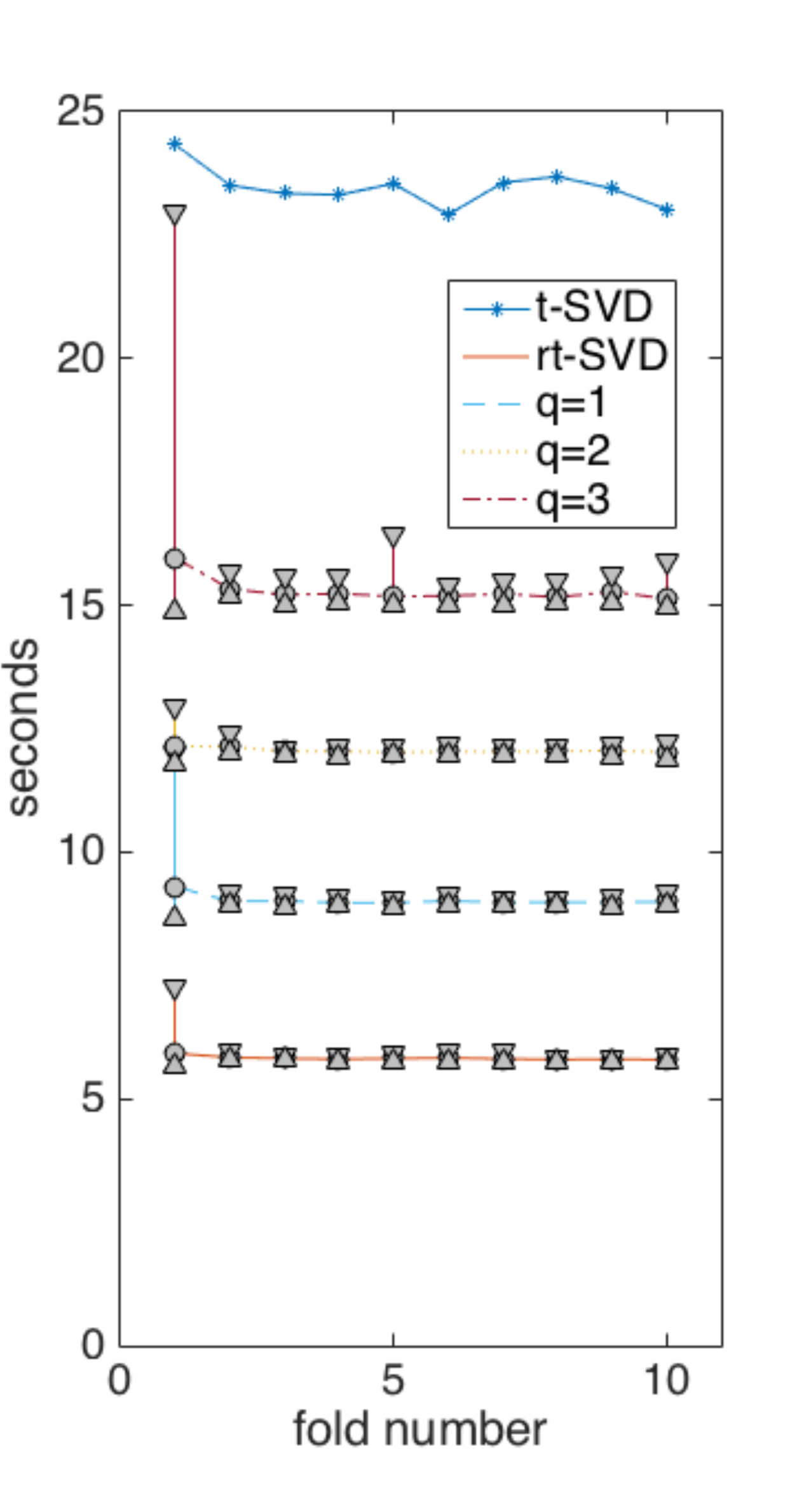}
\caption{(left) Running time to process the the training dataset of Cropped Yale B with $k= 25$. (right) Running time to process the training dataset of Cropped Yale B with $k=50$.}\label{CP_run}
\end{figure}

In Table \ref{rank25_CP}, we make the following observations. 
\begin{itemize}
	\item The minimum and maximum recognition rates are very close to the mean recognition rate in the series of rt-SVD methods.
	\item Comparing the recognition rate between t-SVD and the series of rt-SVD methods, they are identical in $7$ out of $10$ folds. In the other $3$ folds (fold 3, fold 4, and fold 10), the difference is very slight, less than $.001$. 
    \item In fold 10, the maximum recognition rates of the series of rt-SVD method are even slightly higher than the recognition rate of rt-SVD.
	\end{itemize}
	
 The randomized algorithms show almost no variation between different realizations which shows that while there is a probability of failure, however small,  the accuracy of the low-rank representations concentrates about its mean value. Table \ref{rank50_CP} shows similar results as Table \ref{rank25_CP}. In particular, the recognition rates are identical in each fold. The rt-SVD  is about a third as expensive as the full t-SVD.

\subsubsection{AT$\&$T Dataset}
For the AT$\&$T dataset, there are 400 images, so the size of training dataset in each fold is  $112 \times 360 \times 92$ and the size of test dataset in each fold is  $192 \times 40 \times 168$. As compared to the  Cropped Yale B dataset, the AT$\&$T dataset has images of people with different poses, some sample faces are shown in Figure~\ref{attpeople}. As we discussed in subsection \ref{YaleB}, we provide the result with two different target truncation terms, $15$ and $25$. Tables~\ref{rank15_att} and~\ref{rank25_att} show the performance of rt-SVD and rt-SVD with subspace iterations. Figure~\ref{att_run} shows the comparison of running times. The numerical results are consistent with numerical result on the Cropped Yale B dataset, and this demonstrates our algorithms have good performance on both the illumination-varying dataset and pose-varying dataset.

\begin{figure}[!ht]\label{attpeople}
\centering
\includegraphics[scale = .5]{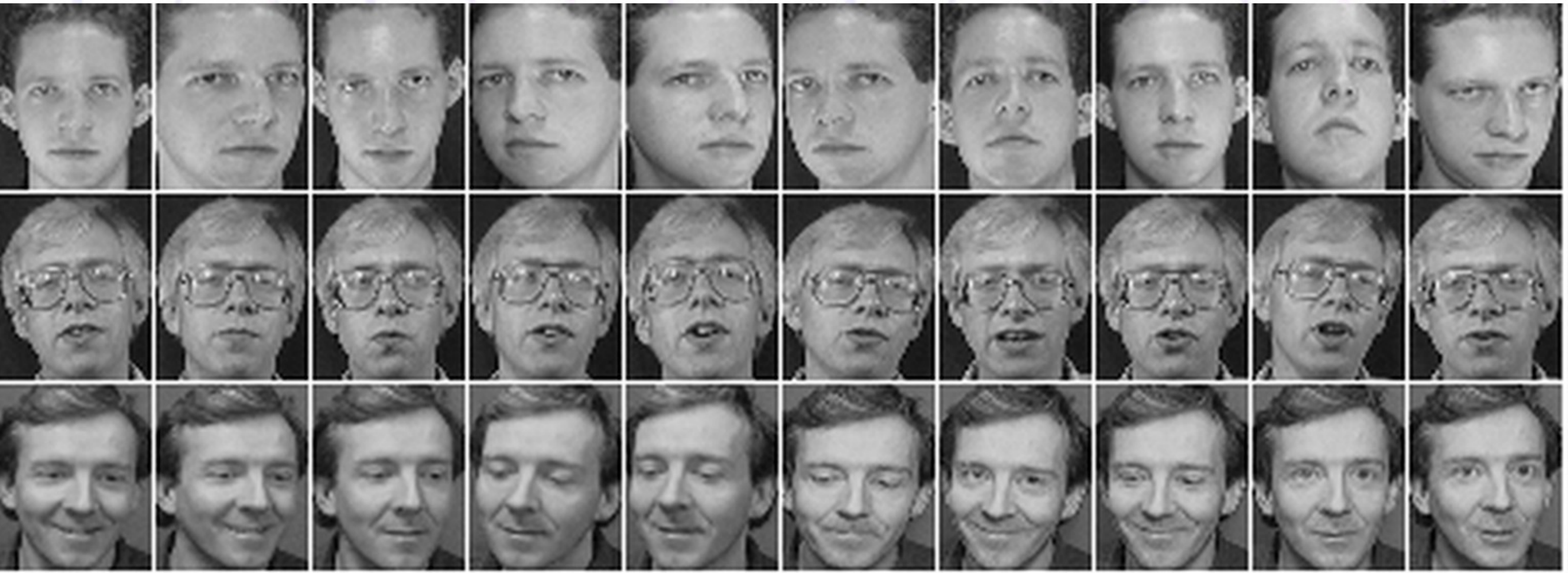}
\caption{Sample images from AT$\&$T dataset.}
\end{figure}

\begin{table}[!ht]
\centering
\caption{Recognition Rates on AT$\&$T dataset with $k= 15$}
\begin{tabular}{|l|l|l|l|l|l|l|l|l|l|l|}
\hline
r & fold 1 & fold 2 & fold 3 & fold 4 & fold 5 & fold 6 & fold 7 & fold 8 & fold 9 & fold 10 \\ \hline
 \multicolumn{11}{|c|}{The t-SVD method}\\
\hline
 &  0.9750 & 1.0000 & 1.0000 & 0.9750 & 0.9750 & 0.9500 & 0.9250 & 1.0000 & 0.9750 & 0.9000 \\
\hline
\multicolumn{11}{|c|}{The rt-SVD method}\\
\hline
mean & 0.9750 & 1.0000 & 1.0000 & 0.9775 & 0.9750 & 0.9537 & 0.9250 & 1.0000 & 0.9750 & 0.9013 \\
min &0.9750 & 1.0000 & 1.0000 & 0.9750 & 0.9750 & 0.9500 & 0.9250 & 1.0000 & 0.9750 & 0.9000 \\
max & 0.9750 & 1.0000 & 1.0000 & 1.0000 & 0.9750 & 0.9750 & 0.9250 & 1.0000 & 0.9750 & 0.9250 \\
 \hline
 \multicolumn{11}{|c|}{The rt-SVD method with subspace iterations $q=1$}\\
\hline
mean & 0.9750 & 1.0000 & 1.0000 & 0.9750 & 0.9750 & 0.9500 & 0.9250 & 1.0000 & 0.9750 & 0.9000 \\
min &  0.9750 & 1.0000 & 1.0000 & 0.9750 & 0.9750 & 0.9500 & 0.9250 & 1.0000 & 0.9750 & 0.9000 \\
max & 0.9750 & 1.0000 & 1.0000 & 0.9750 & 0.9750 & 0.9500 & 0.9250 & 1.0000 & 0.9750 & 0.9000 \\
 \hline
 \multicolumn{11}{|c|}{The rt-SVD method with subspace iterations $q=2$}\\
\hline
mean & 0.9750 & 1.0000 & 1.0000 & 0.9750 & 0.9750 & 0.9500 & 0.9250 & 1.0000 & 0.9750 & 0.9000 \\
min & 0.9750 & 1.0000 & 1.0000 & 0.9750 & 0.9750 & 0.9500 & 0.9250 & 1.0000 & 0.9750 & 0.9000 \\
max &  0.9750 & 1.0000 & 1.0000 & 0.9750 & 0.9750 & 0.9500 & 0.9250 & 1.0000 & 0.9750 & 0.9000 \\
 \hline

\end{tabular}\label{rank15_att}

\end{table}

\begin{table}[!ht]
\centering
\caption{Recognition Rates on AT$\&$T dataset with $k= 25$}
\begin{tabular}{|l|l|l|l|l|l|l|l|l|l|l|}
\hline
r & fold 1 & fold 2 & fold 3 & fold 4 & fold 5 & fold 6 & fold 7 & fold 8 & fold 9 & fold 10 \\ \hline
\multicolumn{11}{|c|}{The t-SVD method}\\
\hline
 & 0.9750 & 1.0000 & 1.0000 & 0.9750 & 0.9500 & 0.9500 & 0.9250 & 1.0000 & 0.9750 & 0.9000 \\
\hline
 \multicolumn{11}{|c|}{The rt-SVD method}\\
\hline
mean &0.9750 & 1.0000 & 1.0000 & 0.9750 & 0.9587 & 0.9500 & 0.9250 & 1.0000 & 0.9750 & 0.9000 \\

min & 0.9750 & 1.0000 & 1.0000 & 0.9750 & 0.9500 & 0.9500 & 0.9250 & 1.0000 & 0.9750 & 0.9000 \\

max & 0.9750 & 1.0000 & 1.0000 & 0.9750 & 0.9750 & 0.9500 & 0.9250 & 1.0000 & 0.9750 & 0.9000 \\

\hline
\multicolumn{11}{|c|}{The rt-SVD method with subspace iterations $q=1$}\\
\hline
mean &  0.9750 & 1.0000 & 1.0000 & 0.9750 & 0.9500 & 0.9500 & 0.9250 & 1.0000 & 0.9750 & 0.9000 \\

min & 0.9750 & 1.0000 & 1.0000 & 0.9750 & 0.9500 & 0.9500 & 0.9250 & 1.0000 & 0.9750 & 0.9000 \\

max &  0.9750 & 1.0000 & 1.0000 & 0.9750 & 0.9500 & 0.9500 & 0.9250 & 1.0000 & 0.9750 & 0.9000 \\

 \hline
 \multicolumn{11}{|c|}{The rt-SVD method with subspace iterations $q=2$}\\
\hline
mean & 0.9750 & 1.0000 & 1.0000 & 0.9750 & 0.9500 & 0.9500 & 0.9250 & 1.0000 & 0.9750 & 0.9000 \\
min & 0.9750 & 1.0000 & 1.0000 & 0.9750 & 0.9500 & 0.9500 & 0.9250 & 1.0000 & 0.9750 & 0.9000 \\
max & 0.9750 & 1.0000 & 1.0000 & 0.9750 & 0.9500 & 0.9500 & 0.9250 & 1.0000 & 0.9750 & 0.9000 \\
\hline

\end{tabular}\label{rank25_att}
\end{table}


\begin{figure}[!ht]
\centering
\includegraphics[scale=.45]{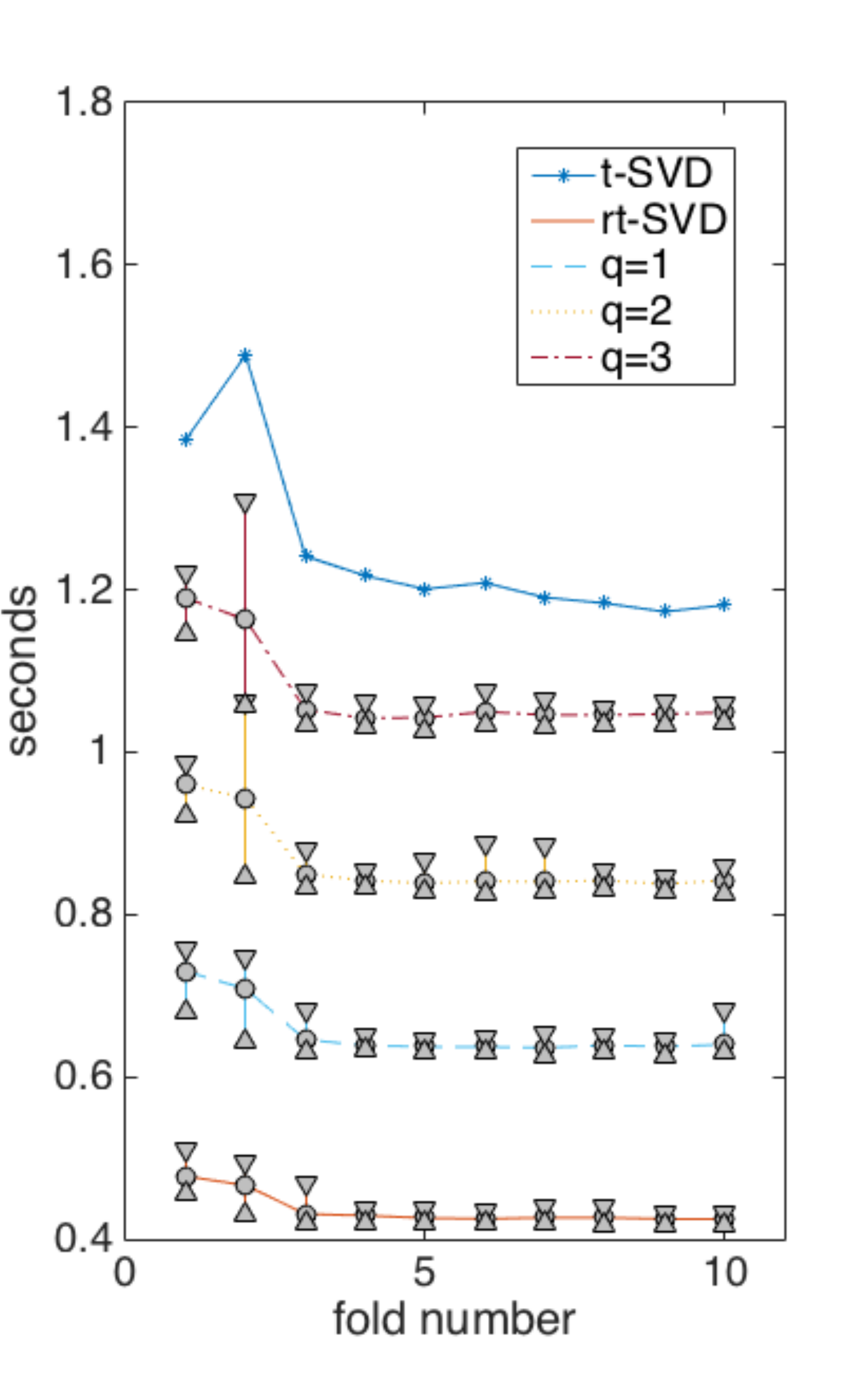}
\includegraphics[scale=.45]{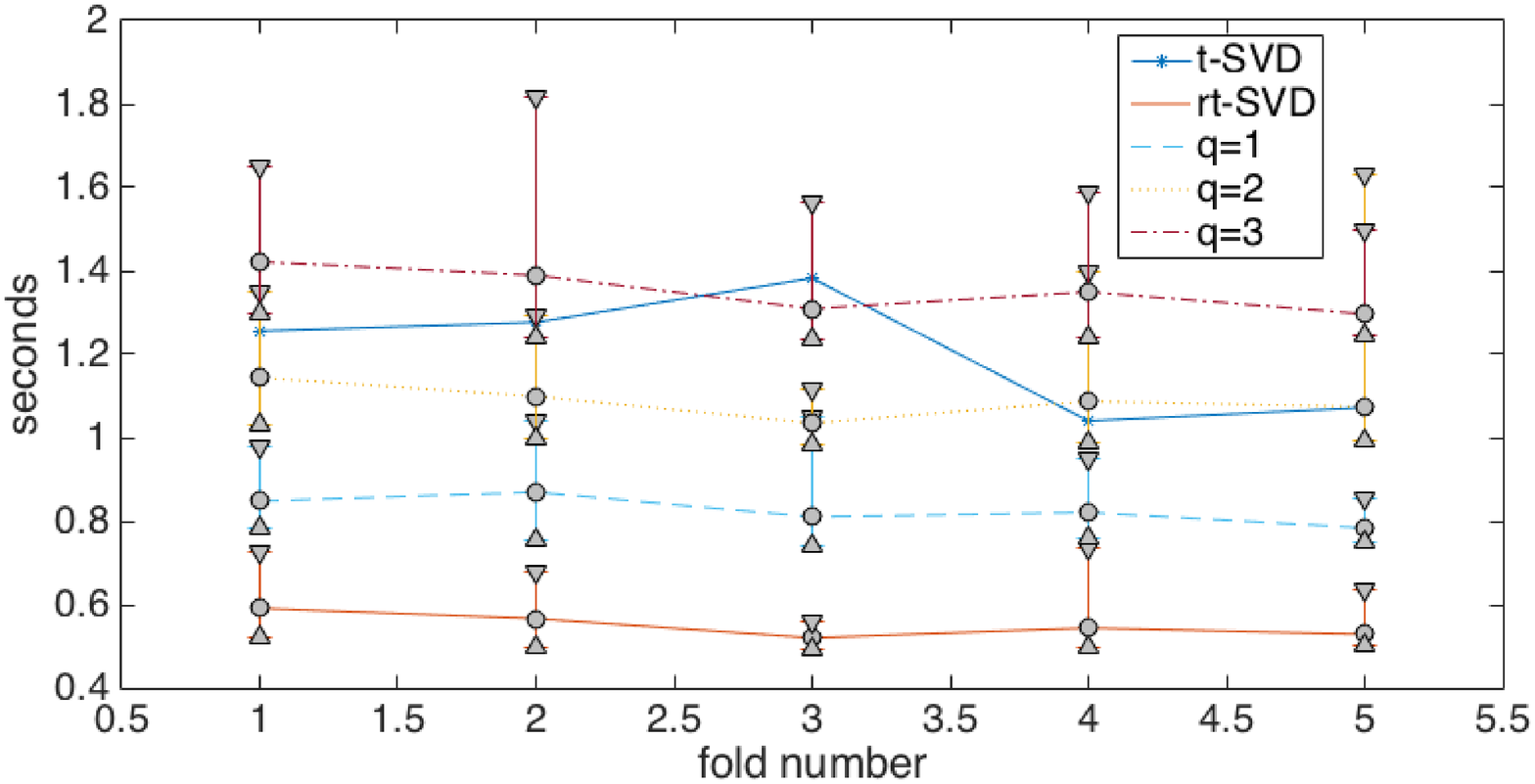}
\caption{(up) Running time to process the $AT\&T$ training dataset with $k= 15$. (down) Running time to process the $AT\&T$ training dataset with $k=25$.}\label{att_run}
\end{figure}


\subsection{Computation time for parallel implementation}
In the facial recognition application, most of the computation time is spent on computing the compression, either exactly or approximately. In this subsection, we report the computation times of computing truncated t-SVD, rt-SVD, and rt-SVD with subspace iterations implemented in parallel on a cluster. The dataset we use, as an example, is the Cropped Yale B dataset $\mathcal{B}$, and the target truncation term (i.e., $k$) is $50$. The experiments are run on Matlab 2015a in (Tufts cluster with  Intel(R) Xeon(R) CPU X5675 running at 3.07 GHz (8 cores)), and up to $8$ processors are used. The t-SVD is computed by Algorithm \ref{alg_truncated_t-svd}, the rt-SVD is computed using Algorithm \ref{alg_rt-svd_fourier}, and the rt-SVD with subspace iterations is computed using Algorithm~\ref{alg_subspace_fourier}. The results are shown in Figure~\ref{fig_cluster}; as can be seen, with increasing number of processors, the runtime using the parallel computing, decreases on average. 


\begin{figure}[!ht]
\begin{center}
\includegraphics[width=\textwidth]{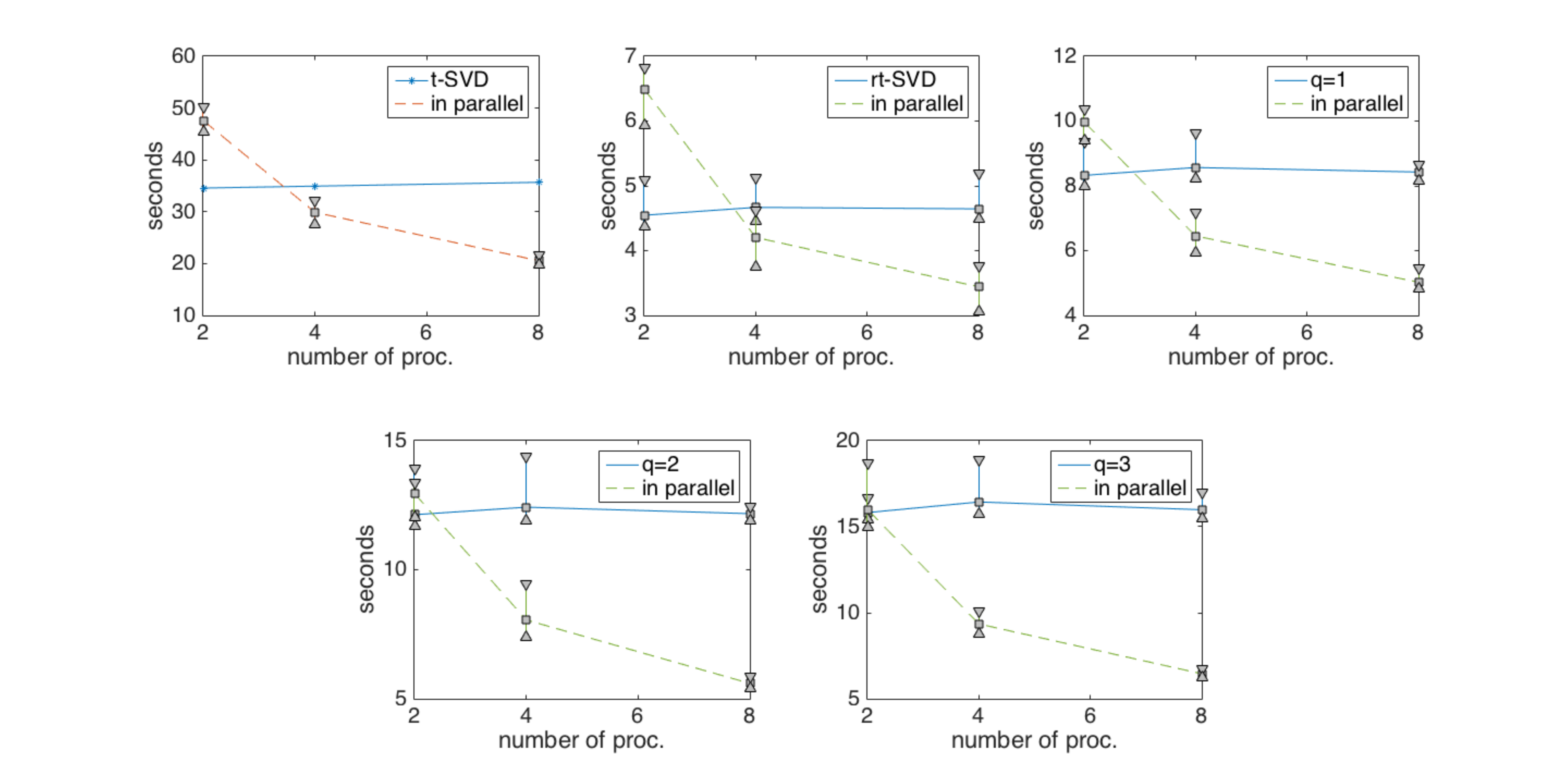}
\end{center}
\caption{The computation time of t-SVD, rt-SVD, and rt-SVD with subspace iterations ($q=1-3$), with and without, parallel computing.}\label{fig_cluster}
\end{figure}

\section{Conclusion and Future Work}  \label{conclusion}

In this paper, we discussed the advantages and limitations of the matrix version of randomized algorithms and deterministic tensor-based algorithms. The algorithms we design combines the advantages of randomization for dimensionality reduction, and applied it to tensor decompositions based on the t-product. We extend the randomized SVD method to third order tensors and provide a basic version algorithm (rt-SVD), as well as a more general version algorithm (rt-SVD with subspace iterations). We showed, theoretically, that the expected errors of both rt-SVD and rt-SVD with subspace iterations, are comparable to the best rank-$k$ approximation obtained using the deterministic t-SVD. Furthermore, we provide numerical support by means of application to facial recognition, on two commonly used publicly available datasets.  Moreover, the randomized algorithms proposed here, can be readily parallelized and we demonstrate the benefits of parallelization on the computational time. This makes our algorithm both accurate and efficient, in practice. Our algorithms also have the added benefit that if the application is run on a distributed memory machine, the approximation factorization can be separately stored on different processors, and can be conveniently used later, without additional computation cost. 
 
For future work, there are several potential research directions. First, the authors in~\cite{hao2013facial} proposed a different method for truncated tensor SVD decomposition, called t-SVD II. This method, while deterministic, has good performance when the singular values decay at different rates across the frontal slices. One avenue of future research would be to develop a  randomized version of t-SVD II, the benefits of which maybe higher recognition rate and lower computational cost. In this paper, we have used Gaussian random tensors for dimensionality reduction. The second possible direction is to investigate other possible structured random tensors, based on subsampled Hadamard or Fourier Transform. 
One drawback of the current methods for compression is that, if a new image or set of images is introduced into the database, then we have to recompute the approximate factorization which can be very expensive. It would be interesting to extend   updating or down-dating of matrix factorizations to tensors using the t-product. Finally, the randomized algorithms can be extended to other tensor decompositions based on invertible linear transforms (see~\cite{kernfeld2015tensor} for more details).

\appendix
\section{Randomized subspace iteration}\label{s_proof}

Our goal in this section is to prove Theorem~\ref{t_gu_subspace}. Before we give the proof, we define the following quantities that will be used later.

 Partition $\mat{A}$ conformally as 
\begin{equation}
\mat{A}\> =  \> \begin{bmatrix} \mat{U}_1 & \mat{U}_2 \end{bmatrix} \begin{bmatrix} \mat{\Sigma}_1 & \\ & \mat{\Sigma}_2 \end{bmatrix}  \begin{bmatrix} \mat{V}_1^{\rm H} \\ \mat{V}_2^{\rm H} \end{bmatrix} \,.
\end{equation}
Define $\mat{W}_1 \equiv \mat{V}_1^{\rm H}\mat{W}$ and $\mat{W}_2 \equiv \mat{V}_2^{\rm H}\mat{W}$. The subspace iteration (Algorithm~\ref{alg_rSVD_power_subspace}) with random starting guess, can be alternatively expressed as  \[ \mat{Y} \> =\>  (\mat{AA}^{\rm H})^q\mat{AW} \> = \> \mat{U} \begin{bmatrix} \mat{\Sigma}_1^{2q+1} & \\ & \mat{\Sigma}_2^{2q+1} \end{bmatrix} \begin{bmatrix} \mat{W}_1 \\ \mat{W}_2 \end{bmatrix}.\]

Suppose $\mat{V}$ has full column rank; we denote the projection matrix by $\proj{V} \equiv \mat{V}(\mat{V}^{\rm H}\mat{V})^{-1}\mat{V}^{\rm H}$ corresponding to $\range(V)$. Additionally, because of $\mat{V}$ has orthonormal columns, then the expression for the projector simplifies to $\proj{V} = \mat{VV}^{\rm H}$.

We present a result that characterizes the error of the low-rank matrix approximation. The result makes minimal assumptions about the sampling matrix $\mat{W}$, and is therefore, applicable to any distribution. This result will be used to prove Theorem~\ref{t_gu_subspace}. 
\begin{theorem}\label{t_struct}
Let $\mat{A}$ be an $m\times n$ matrix with SVD $\mat{A} = \mat{U\Sigma V}^{\rm H}$ and  $k\geq 0$ be a fixed parameter. Choose a matrix $\mat{W} \in \mathbb{C}^{n\times \ell}$, define $\mat{W}_1$ and $\mat{W}_2$ as above and assume that $\mat{W}_1$ has full row rank. Compute $\mat{Q}$ using Algorithm~\ref{t_gu_subspace}. Then the approximation error satisfies
\begin{equation}
\normf{(\mat{I}-\mat{QQ}^{\rm H})\mat{A}}^2 \> \leq \> \normf{\mat{\Sigma}_2}^2 + \tau_k^{4q}\normf{\mat{\Sigma}_2\mat{W}_2\mat{W}_1^\dagger}^2.
\end{equation}
\end{theorem}
\begin{proof}
Define matrices $\mat{Z}$ and $\mat{F}$ as 
\[ \mat{Z} \> =\> \mat{U}^{\rm H} \mat{Y} \mat{W}_1^\dagger \mat{\Sigma}_1^{-(2q+1)} = \begin{bmatrix} \mat{I} \\ \mat{F} \end{bmatrix} \qquad \mat{F} \equiv \mat{\Sigma}_2^{2q+1} \mat{W}_2\mat{W}_1^\dagger\mat{\Sigma}_1^{-(2q+1)} .  \]
The construction of $\mat{Z}$ ensures that the following results hold 
\begin{equation}\label{e_range}
 \range(\mat{Z}) \subset \range(\mat{Y}) = \range(\mat{U}^{\rm H} \mat{Y}) = \range(\mat{U}^{\rm H}\mat{Q}).
\end{equation}
 Plugging the SVD of $\mat{A}$ into the error in the low-rank approximation
\[\normf{(\mat{I}-\mat{QQ}^{\rm H})\mat{A}}^2 \> = \> \normf{(\mat{I}-\proj{Q})\mat{U\Sigma V}^{\rm H}}^2 =  \normf{(\mat{I}-\proj{Q})\mat{U\Sigma}}^2,\]
where the last step follows since the Frobenius norm is unitarily invariant. Recall that~\cite[Proposition 8.4]{2011halko} implies that $\mat{U}^{\rm H}\proj{Q}\mat{U} = \proj{\mat{U}^{\rm H}\mat{Q}}$, and from~\cite[Proposition 8.5]{2011halko} and~\eqref{e_range} follows 
\[ \normf{(\mat{I}-\proj{Q})\mat{U\Sigma}}^2 = \normf{\mat{\Sigma}^{\rm H} (\mat{I}-\proj{U^{\rm H} Q})\mat{\Sigma}} \leq \normf{\mat{\Sigma}^{\rm H} (\mat{I}-\proj{Z})\mat{\Sigma}} = \normf{(\mat{I}-\proj{Z})\mat{\Sigma}}^2.\] 
Following the steps of~\cite[Theorem 9.1]{2011halko} it can be shown that 
\[ (\mat{I}-\proj{Z})\mat{\Sigma}  = \begin{bmatrix}  (\mat{I}+\mat{F}^{\rm H}\mat{F})^{-1}\mat{F}^{\rm H}\mat{F}\mat{\Sigma}_1 \\ \left( \mat{I}-\mat{F}(\mat{I}+\mat{F}^{\rm H}\mat{F})^{-1}\mat{F}^{\rm H}\right)\mat{\Sigma}_2\end{bmatrix}. \]
We also recall from~\cite[Theorem 9.1]{2011halko} the following inequality
\begin{equation}\label{e_pert_ineq}
  \mat{I}-\mat{F}(\mat{I}+\mat{F}^{\rm H}\mat{F})^{-1}\mat{F}^{\rm H} \> \preceq \> \mat{I}. 
\end{equation}
We can then bound
\begin{equation}\label{e_projz} 
\normf{(\mat{I}-\proj{Z})\mat{\Sigma}}^2 \> = \> \normf{(\mat{I}+\mat{F}^{\rm H}\mat{F})^{-1}\mat{F}^{\rm H}\mat{F}\mat{\Sigma}_1}^2 + \normf{ \left( \mat{I}-\mat{F}(\mat{I}+\mat{F}^{\rm H}\mat{F})^{-1}\mat{F}^{\rm H}\right)\mat{\Sigma}_2}^2 \\
\end{equation}
where the last step follows from~\eqref{e_pert_ineq} and~\cite[Proposition 8.4]{2011halko}. We can bound
\[   \normf{(\mat{I}+\mat{F}^\top\mat{F})^{-1}\mat{F}^{\rm H}\mat{F}\mat{\Sigma}_1}\> \leq \> \| \mat{F}(\mat{I}+\mat{F}^{\rm H}\mat{F})^{-1}\|_2 \normf{\mat{F\Sigma}_1} \> \leq \> \normf{\mat{F\Sigma}_1}, \]
where the second result follows from the sub-multiplicativity of the Frobenius norm, and the last result follows from a simple SVD argument. Together with~\eqref{e_projz} this gives  
\begin{equation}
\label{e_projz2}
 \normf{(\mat{I}-\proj{Q})\mat{A}}^2 \> \leq  \>\normf{(\mat{I}-\proj{Z})\mat{\Sigma}}^2 \> \leq  \> \normf{\mat{\Sigma}_2}^2 + \normf{\mat{F\Sigma}_1}^2. 
\end{equation}
Observe that $\mat{F\Sigma}_1 = \mat{\Sigma}_2^{2q} (\mat{\Sigma}_2\mat{W}_2\mat{W}_1^\dagger) \mat{\Sigma}_1^{-2q}$, and together with some simple norm inequalities, it follows that 
\[ \normf{\mat{F\Sigma}_1} \> \leq\> \| \mat{\Sigma}_2^{2q}\|_2 \| \mat{\Sigma}_1^{-2q}\|_2 \normf{\mat{\Sigma}_2\mat{W}_2\mat{W}_1^\dagger} \> \leq \>\tau_k^{2q}\normf{\mat{\Sigma}_2\mat{W}_2\mat{W}_1^\dagger}. \]
Combining this result with~\eqref{e_projz2} gives the desired result.
\end{proof}

We are now ready to prove Theorem~\ref{t_gu_subspace}.

\textit{Proof of Theorem~\ref{t_gu_subspace}}
Apply the same arguments as~\cite[Theorem 10.5]{2011halko} to Theorem~\ref{t_struct}. 
$\square$

\begin{remark}\label{r_gucomp}
The estimate in Theorem~\ref{t_gu_subspace} is sharper than~\cite[Theorem 5.7]{gu2015subspace} when $\ell = k + p$ and $p \geq 2$. This implies that the singular value gap is chosen to be between $k$ and $k+1$. To see this, we can simplify the bound in Theorem~\ref{t_gu_subspace} to
\[ \mathbb{E}\normf{ \mathbf{A}-\mathbf{Q}\mathbf{Q}^{\rm H}  \mathbf{A}} \> \leq
\> \sqrt{\sum_{j > k} \sigma_j^2 + \frac{k\tau_k^{4q}(n-k)}{p-1}\sigma_{k+1}^2}. \]
The equivalent bound from~\cite[Theorem 5.7]{gu2015subspace} is 
\[ \mathbb{E}\normf{ \mathbf{A}-\mathbf{Q}\mathbf{Q}^{\rm H}  \mathbf{A}} \> \leq
\> \sqrt{\sum_{j > k} \sigma_j^2 + {k\tau_k^{4q}(n-k)}C^2\sigma_{k+1}^2}, \]
where $C=(\sqrt{n-k}+\sqrt{k+p}+7)(\frac{4e\sqrt{k+p}}{p+1})$. Since the first terms in both expressions are the same, comparing the second terms we have 
\[ \frac{(\sqrt{n-k}+\sqrt{k+p}+7)^2}{n-k}\left(\frac{16e^2(k+p)}{(p+1)^2}\right) (p-1)  \> > \> 1.\]
This shows that our bound is tighter under these conditions. For any other choice of $p$ and $\ell$,~\cite[Theorem 5.7]{gu2015subspace} maybe sharper depending on the singular value decay.    
\end{remark}

\section{Proofs of Section~\ref{rt-svd}}\label{app_tensor}

We present the following Lemma, which is similar to Parseval's theorem. 
\begin{lemma}\label{l_parseval}
Let $\mc{A}$ be a real $n_1\times n_2 \times n_3$ tensor. Then 
\[ \normf{\mc{A}}^2  \> = \> \frac{1}{n_3} \sum_{i=1}^{n_3}\normf{\hat{\mc{A}}^{(i)}}^2,\]
 where $ \hat{\mc{A}}^{(i)} \equiv \hat{\mc{A}}(:,:,i)$ is the $i^{th}$ frontal slice of the tensor in the Fourier domain.
\end{lemma}
\begin{proof}
Define the operation $\text{circ}\left(\text{MatVec}(\mc{A})\right)$ as 
\[ \text{circ}\left(\text{MatVec}(\mc{A})\right)= \begin{bmatrix} \mc{A}^{(1)} & \mc{A}^{(n_3)} & \dots & \mc{A}^{(2)} \\ \mc{A}^{(2)} & \mc{A}^{(1)} & \dots & \mc{A}^{(3)} \\ \vdots & \vdots & \ddots & \vdots \\ \mc{A}^{(n_3)} & \mc{A}^{(n_3-1)} & \dots & \mc{A}^{(1)} \end{bmatrix}.\]
Let $\mat{F}_{n_3}$ be the Discrete Fourier Matrix. Then~\cite[Equation 3.1]{2011kilmer} implies that 
\[ (\mat{F}_{n_3} \otimes \mat{I}_{n_1}) \text{circ}\left(\text{MatVec}(\mc{A})\right) (\mat{F}_{n_3}^{\rm H} \otimes \mat{I}_{n_2}) = \begin{bmatrix} \hat{\mc{A}}^{(1)} & \\ & \hat{\mc{A}}^{(2)} & \\ & & \ddots & \\ &&&\hat{\mc{A}}^{(n_3)}\end{bmatrix}.  \]

The unitary invariance of the Frobenius norm implies 
\[ \normf{(\mat{F}_{n_3} \otimes \mat{I}_{n_1}) \text{circ}\left(\text{MatVec}(\mc{A})\right) (\mat{F}_{n_3}^{\rm H} \otimes \mat{I}_{n_2})}^2 = \sum_{i=1}^{n_3}\normf{\hat{\mc{A}}^{(i)}}^2. \]
Alternatively, from its definition it follows that 
\[ \normf{\mc{A}}^2 = n_3\sum_{i=1}^{n_3}\normf{\mc{A}^{(i)}}^2 = \normf{(\mat{F}_{n_3} \otimes \mat{I}_{n_1}) \text{circ}\left(\text{MatVec}(\mc{A})\right) (\mat{F}_{n_3}^{\rm H} \otimes \mat{I}_{n_2})}^2  .\]
Equating appropriate terms gives us the desired result.
\end{proof}

We now use this result to prove Theorems~\ref{expectation} and~\ref{expectation_subspace}.

\textit{Proof of Theorem~\ref{expectation}}
Using Lemma~\ref{l_parseval} and the linearity of the expectation, we can write 
\begin{equation}\label{proof2}
	\mathbb{E}\normf{ \mathcal{A}-\mathcal{Q}*\mathcal{Q}^{\rm T} \ast \mathcal{A}}^2 \> \leq
	 \frac{1}{n_3}\left( \sum_{i=1}^{n_3} \mathbb{E}\, \normf{ \hat{\mathcal{A}}^{(i)} - \hat{\mathcal{Q}}^{(i)}(\hat{\mathcal{Q}}^{(i)})^{\rm H}\hat{\mathcal{A}}^{(i)}}^2  \right).
\end{equation}
We can bound the individual terms in the summation by applying the result from~\eqref{expected_error}, 
\begin{equation}\label{proof3}
	\mathbb{E}\normf{ \hat{\mathcal{A}}^{(i)} - \hat{\mathcal{Q}}^{(i)}(\hat{\mathcal{Q}}^{(i)})^\top\hat{\mathcal{A}}^{(i)} }^2 \leq \left({1+\dfrac{k}{p-1}} \right)\,\left(\sum_{j>k} (\hat{\sigma}^{(i)}_{j})^{2}\right).
\end{equation}
Substitute inequality~\eqref{proof3} into inequality~\eqref{proof2} to obtain
\[	\mathbb{E}\normf{ \mathcal{A}-\mathcal{Q}*\mathcal{Q}^{\rm T} }^2 \> \leq \>  \left({1+\dfrac{k}{p-1}} \right)\,\left(\frac{1}{n_3}\sum_{i=1}^{n_3}\sum_{j>k} (\hat{\sigma}^{(i)}_{j})^{2}\right). \]
Finally, using H\"{o}lder's inequality
 \[ 	\mathbb{E}\normf{ \mathcal{A}-\mathcal{Q}*\mathcal{Q}^{\rm T} \ast \mathcal{A}} \leq  \left(	\mathbb{E}\normf{ \mathcal{A}-\mathcal{Q}*\mathcal{Q}^{\rm T}\ast \mathcal{A} }^2 \right)^{1/2},\]
from which the desired result follows.
$\square$

\textit{Proof of Theorem~\ref{expectation_subspace}}
Apply the results of Theorem~\ref{t_gu_subspace} to each frontal slice in~\eqref{proof2} and follow the remaining steps of Theorem~\ref{expectation}. 
$\square$

We next prove Theorem~\ref{concentration}. We will need a result from~\cite{gu2015subspace}.
\begin{lemma}\cite[Theorem 5.8]{gu2015subspace}\label{l_conc}
Let $\mat{W}_1$ and $\mat{W}_2$ be defined as in Appendix~\ref{s_proof}. Let $0 < \delta < 1$ be the failure probability and define the constant
\[ C_\delta = \frac{e\sqrt{k+p}}{p+1} \left(\frac{2}{\delta}\right)^{\frac{1}{p+1}} \left(\sqrt{n-k} + \sqrt{k+p} + \sqrt{2\log \frac{2}{\delta}}\right).\]
Then $ \text{Prob}\left\{\|\mat{W}_2 \|_2 \| \mat{W}_1^\dagger\|_2 \geq C_\delta \right\} \> \leq \> \delta. $
\end{lemma}

\textit{Proof of Theorem~\ref{concentration}}
Apply Lemma~\ref{l_parseval} and Theorem~\ref{t_struct} to obtain
\[ \normf{ \mathcal{A}-\mathcal{Q}*\mathcal{Q}^{\rm T} \ast \mathcal{A}}^2 \> \leq \frac{1}{n_3}\sum_{i=1}^{n_3}\left(\normf{\hat{\mc{S}}_2^{(i)}}^2 + (\tau_k^{(i)})^{4q} \normf{\hat{\mc{S}}_2^{(i)}\hat{\mc{W}}_2^{(i)}(\hat{\mc{W}}_1^{(i))})^\dagger}^2\right) .
\]
The sub-multiplicative property of the Frobenius norm implies $$\normf{\hat{\mc{S}}_2^{(i)}\hat{\mc{W}}_2^{(i)}(\hat{\mc{W}}_1^{(i))})^\dagger} \leq \left\|\hat{\mc{W}}_2^{(i)}(\hat{\mc{W}}_1^{(i))})^\dagger \right\|_2\normf{\hat{\mc{S}}_2^{(i)}}.$$ Plugging this into the above equation gives us 

\[ \normf{ \mathcal{A}-\mathcal{Q}*\mathcal{Q}^{\rm T} \ast \mathcal{A}}^2 \> \leq \> \frac{1}{n_3}\sum_{i=1}^{n_3}\left(1 + (\tau_k^{(i)})^{4q} \left\|\hat{\mc{W}}_2^{(i)}(\hat{\mc{W}}_1^{(i))})^\dagger\right\|^2\right)\left(\sum_{j>k} (\hat{\sigma}^{(i)}_{j})^{2}\right).
\]
Since Gaussian random matrices are invariant to rotations, $\mc{W}_2^{(i)}$ and $\mc{W}_1^{(i)}$ will be independent for each $i=1,\dots,n_3$. We can therefore apply Lemma~\ref{l_conc} to obtain the desired result.
$\square$

\bibliographystyle{plain}	


\bibliography{bibfile}		

\end{document}